\documentclass[10pt,journal,twocolumn,final,twoside]{IEEEtran}
\usepackage{amsmath, amssymb, amsthm, stfloats, amsfonts,cite}
\usepackage{graphicx}
\usepackage{bm,dsfont,subfigure}
\usepackage[bookmarks=false]{hyperref}
\usepackage{empheq}
\usepackage{algorithmic,float}  
\usepackage{algorithm}
\usepackage{colortbl}
\usepackage[flushleft]{threeparttable}
\usepackage{color}
\usepackage{stfloats} 
\usepackage{nicefrac}
\usepackage{enumerate}
\usepackage{multirow}% http://ctan.org/pkg/multirow
\newcommand{\sw}{{\scriptstyle{\mathcal{W}}}}
\newcommand{\sLambda}{\mathcal{D}} 
\usepackage{pifont}% http://ctan.org/pkg/pifont
\newcommand{\cmark}{\ding{51}}%
\newcommand{\xmark}{\ding{53}}%
\DeclareMathAlphabet{\mathbbold}{U}{bbold}{m}{n}
\makeatletter 
\newcases{mycases}{\quad}{%
  \hfil$\m@th\displaystyle{##}$}{$\m@th\displaystyle{##}$\hfil}{\lbrace}{.}
\makeatother
\def\E{{\mathbb E}}
\def\T{{\mathsf T}}
\def\w{{\boldsymbol \sw}}
\def\blambda{{\boldsymbol \lambda}}

\def\v{{\boldsymbol v}}
\def\x{{\boldsymbol x}} 
\def\d{{\boldsymbol d}}
\def\h{{\boldsymbol h}}
\def\u{{\boldsymbol u}}

\def\z{{\boldsymbol z}}
\def\Tr{{\mathrm{Tr}}}
\def\det{{\mathrm{det}}}
\renewcommand{\qed}{\hfill$\blacksquare$}

\newtheorem{theorem}{Theorem}
\newtheorem{lemma}{Lemma}
\newtheorem{corollary}{Corollary}
\newtheorem{definition}{Definition}

\newtheorem{example}{Example}

\title{Stability and Performance Limits of\\ Adaptive Primal-Dual Networks}
\allowdisplaybreaks
\begin{document}


\author{Zaid~J.~Towfic,~\IEEEmembership{Member,~IEEE,}
        and~Ali~H.~Sayed,~\IEEEmembership{Fellow,~IEEE}% <-this % stops a space
        
\thanks{
This work was supported in part by NSF grants CCF-1011918 and ECCS-1407712. A short version of this work appears in \cite{towfic2015distributedprimal}.}
\thanks{
Zaid J. Towfic is with MIT Lincoln Laboratory, Lexington, MA. Email: ztowfic@ucla.edu.
}
\thanks{
A. H. Sayed is with the Department of Electrical Engineering,
University of California, Los Angeles, CA 90095. Email: sayed@ee.ucla.edu.
This work was completed while Z. J. Towfic was a PhD student at UCLA.
}% <-this % stops a space
%\thanks{Color versions of one or more of the figures in this paper are available online
%at http://ieeexplore.ieee.org.}%
}

\maketitle

\begin{abstract}
	This work studies  distributed primal-dual strategies for adaptation and learning over networks from streaming data. Two first-order methods are considered  based on the Arrow-Hurwicz (AH) and augmented Lagrangian (AL) techniques. Several revealing results are discovered in relation to the performance and stability of these strategies when employed over adaptive networks. The conclusions establish that the advantages that these methods exhibit for deterministic optimization problems do not necessarily carry over to stochastic optimization problems. It is found that they have narrower stability ranges and worse steady-state mean-square-error performance than primal methods of the consensus and diffusion type. It is also found that the AH technique can become unstable under a partial observation model, while the other techniques are able to recover the unknown under this scenario. A method to enhance the performance of AL strategies is proposed by tying the selection of the step-size to their regularization parameter.  It is shown that this method allows the AL algorithm to approach the performance of consensus and diffusion strategies but that it remains less stable than these other strategies.
\end{abstract}


\section{Introduction}
\IEEEPARstart{D}{istributed} estimation is the task of estimating and tracking slowly drifting parameters by a network of agents, based solely on local interactions. In this work, we focus on distributed strategies that enable {\em continuous} adaptation and learning from streaming data by relying  on stochastic gradient updates that employ {\em constant} step-sizes. The resulting networks become adaptive in nature, which means that the effect of gradient noise never dies out and seeps into the operation of the algorithms.  For this reason, the design of such networks requires careful analysis in order to assess performance and provide convergence guarantees.

Many efficient algorithms have already been proposed in the literature for inference over networks  \cite{Li20082599,Krishnamurthy,Takahashi2008139,tsianos2012distributed,cassio,Cattivelli10,NOW_ML,ram2010distributed,chen2011TSPdiffopt,theodoridis2011adaptive,chouvardas2011adaptive,nedic2009distributed,dimakis2010gossip,kar2011converegence, schizas2009distributed,rabbat2005quantized,dini2012cooperative,Scaglione} such as consensus  strategies \cite{nedic2009distributed,dimakis2010gossip,kar2011converegence, schizas2009distributed} and diffusion strategies \cite{cassio,Cattivelli10,ram2010distributed,chen2011TSPdiffopt,NOW_ML}, \cite{chouvardas2011adaptive,SayedProcIEEE,sayed2013diffusion,sayed2012diffbookchapter}. These strategies belong to the class of {\em primal} optimization techniques since they rely on estimating and propagating the primal variable. Previous studies have shown that sufficiently small step-sizes enable these strategies to learn well and in a stable manner. Explicit conditions on the step-size parameters for mean-square stability, as well as closed-form expressions for their steady-state mean-square-error performance already exist (see, e.g., \cite{SayedProcIEEE,NOW_ML} and the many references therein). 
Besides primal methods, in the broad optimization literature, there is a second formidable class of techniques known as {\em primal-dual}  methods such as the Arrow-Hurwicz (AH) method \cite{arrow1958studies,poliak1987introduction} and the augmented Lagrangian (AL) method \cite{poliak1987introduction,bertsekas1999nonlinear}. These methods rely on propagating two sets of variables: the primal variable and a dual variable. The main advantage relative to primal methods, for deterministic optimization, is their ability to avoid ill-conditioning when solving constrained problems \cite[pp.~244-245]{poliak1987introduction} \cite[pp.~630-631]{griva2009linear}.

In contrast to existing useful studies on primal-dual algorithms (e.g., \cite{Paolo,boyd2011distributed}), we shall examine this class of strategies in the context of adaptive networks, where the optimization problem is \emph{not} necessarily static anymore (i.e., its minimizer can drift with time) and where the exact form of the cost function is \emph{not} even known. To do so, we will need to develop distributed variants that can learn directly and continuously from streaming data when the statistical distribution of the data is unknown.  It turns out that under these conditions, the dual function cannot be determined explicitly any longer, and, consequently, the conventional computation of the optimal primal and dual variables cannot assume  knowledge of the dual function.  We will address this difficulty by employing {\em constant} step-size adaptation and {\em instantaneous} data measurements to approximate the search directions. When this is done, the operation of the resulting algorithm becomes influenced by gradient noise, which measures the difference between the desired search direction and its approximation. This complication alters the dynamics of primal-dual techniques in non-trivial ways and leads to some surprising patterns of behavior in comparison to primal techniques. 
Before we comment on these findings and their implications, we remark that the stochastic-gradient versions that we develop in this work can be regarded as a first-order variation of the useful algorithm studied in \cite[p.~356]{Paolo} with one key difference; this reference assumes that the cost functions are known exactly to the agents and that, therefore, the dual function can be determined explicitly. In contrast, we cannot make this assumption in the adaptive context. 

One of the main findings in this article is that the studied {\em adaptive} primal-dual strategies turn out to have a smaller stability range and degraded performance in comparison to consensus and diffusion strategies.  This result implies that AH and AL techniques are not as effective for adaptation and learning from streaming data as the primal versions that are based on consensus and diffusion constructions. As explained further ahead, one main reason for this anomaly is that the distributed AH and AL strategies enlarge the state dimension of the network in comparison to primal strategies and they exhibit an asymmetry in their update relations; this asymmetry can cause an unstable growth in the state of the respective networks. In other words, the advantages of AH and AL techniques that are well-known in the static optimization context do not necessarily carry over to the stochastic context. We also remark that in contrast to the earlier work \cite{schizas2009distributed}, we do not require the availability of special {\em bridge} nodes in addition to the regular nodes. In our formulation, the focus is on networks that consist of homogeneous nodes with similar capabilities and responsibilities and without imposing constraints on the network topology. On the other hand, the articles \cite{mateos2009performance,mateos2012distributed} study a variant of the LMS problem, which is studied in this work and \cite{schizas2009distributed}, that is based on a stochastic alternating direction method of multipliers (ADMM) and RLS schemes as opposed to the AH and AL methods. In contrast to \cite{mateos2009performance,mateos2012distributed}, our focus is not just on the theoretical steady-state performance or tracking performance in the midst of link noise, but on the stability of the AH and AL algorithm as a function of the number of nodes in the network as well as the critical role that the augmentation of the Lagrangian plays in the AL algorithm.

A second important conclusion relates to the behavior of AH and AL strategies under the partial observation model. This model refers to the situation in which some agents may not be able to estimate the unknown parameter on their own, whereas the aggregate  information from across the entire network is sufficient for the recovery of the unknown vector through local cooperation.  We discover that the AH strategy can fail under this condition, i.e., the AH network can fail to recover the unknown and become unstable even though the network has sufficient information to allow the agents to arrive at the unknown. This is a surprising conclusion and it can be illustrated analytically by means of examples. In comparison, we show that the AL, consensus, and diffusion strategies are able to recover the unknown under the partial observation model. 

We further find that the stability range for the AL strategy depends on two factors: the size of its regularization parameter and the network topology.  This means that even if all individual agents are stable and able to solve the inference task on their own, the condition for stability of the AL strategy will still depend on how these agents are connected to each other. A similar behavior is exhibited by consensus networks \cite{tu2012diffusion}. This property is a disadvantage in relation to diffusion strategies whose stability ranges have been shown to be independent of the network topology \cite{tu2012diffusion,SayedProcIEEE}. 

We also examine the steady-state mean-square-deviation (MSD) of the primal-dual adaptive strategies and discover that the AH method achieves the same MSD performance as non-cooperative processing. This is a disappointing property since the algorithm employs cooperation, and yet the agents are not able to achieve better performance. On the other hand, the AL algorithm improves on the performance of non-cooperative processing, and can be made to approach the performance of diffusion and consensus strategies as the regularization parameter is increased. This means that the AL algorithm must utilize very small step-sizes to approach the performance that other distributed algorithms can achieve with reasonable parameter values.

\textbf{Notation}: Random quantities are denoted in boldface. The notation $\otimes$ represents the Kronecker product while the notation $\otimes_b$ denotes the Tracy-Singh block Kronecker product \cite{TracySingh}. Throughout the manuscript, all vectors are column vectors with the exception of the regression vector $u$, which is a row vector. Matrices are denoted in capital letters, while vectors and scalars are denoted in lowercase letters. Network variables that aggregate variables across the network are denoted in calligraphic letters.

\section{Adaptive Primal Strategies}
In this section, we describe the problem formulation and review the two main primal techniques: diffusion and consensus for later user. Thus, consider a connected network of $N$ agents that wish to estimate a real $M \times 1$ parameter vector $w^o$ in a distributed manner. Each agent $k=1,2,\ldots,N$  has access to real scalar observations $\d_k(i)$ and zero-mean real $1 \times M$ regression vectors $\u_{k,i}$ that are assumed to be related via the model:
\begin{align}
	\d_k(i) = \u_{k,i} w^o + \v_k(i) \label{eq:ProblemFormulation:LinearModel}
\end{align}
where $\v_k(i)$ is zero-mean real scalar random noise, and $i$ is the time index. Models of the form \eqref{eq:ProblemFormulation:LinearModel} arise in many useful contexts such as in applications involving channel estimation, target tracking, equalization, beamforming, and localization \cite{Sayed08,sayed2013diffusion,sayed2012diffbookchapter}.  We denote the second-order moments by 
\begin{equation}
	R_{u,k} \!=\! \E \u_{k,i}^\T \u_{k,i},\quad\!\!	r_{du,k} \!=\! \E \u_{k,i}^\T \d_k(i),\quad\!\! \sigma_{v,k}^2 \!=\! \E \v_{k}^2(i)  \label{eq:R_u,k}
\end{equation}
and assume that the regression and noise processes are each temporally and spatially white. We also assume that $\u_{k,i}$ and $\v_{\ell}(j)$ are independent of each other for all $k,\ell$ and $i,j$. We allow for the possibility that some individual covariance matrices, $R_{u,k}$, are singular but assume that the sum of all covariance matrices across the agents is positive-definite:
\begin{align}
	\sum_{k=1}^N R_{u,k} > 0 \label{eq:partial_observability}
\end{align}
This situation corresponds to the {\em partial observation} scenario where some of the agents may  not able to solve the estimation problem on their own, and must instead cooperate with other nodes in order to estimate $w^o$.

To determine $w^o$, we consider an optimization problem involving an aggregate mean-square-error cost function:
\begin{align}
	\min_w \frac{1}{2}\sum_{k=1}^N \E (\d_k(i) - \u_{k,i} w)^2 \label{eq:formulation_non_distributed_opt}
\end{align}

\noindent It is straightforward to verify that $w^o$ from \eqref{eq:ProblemFormulation:LinearModel} is the unique minimizer of \eqref{eq:formulation_non_distributed_opt}. Several useful algorithms have been proposed in the literature to solve \eqref{eq:formulation_non_distributed_opt} in a distributed manner. We are particularly interested in adaptive algorithms that operate on streaming data and do not require knowledge of the underlying signal statistics.  Some of the more prominent methods in this regard are consensus-type strategies  \cite{nedic2009distributed,dimakis2010gossip,kar2011converegence,schizas2009distributed} and diffusion strategies \cite{cassio,Cattivelli10,sayed2013diffusion,sayed2012diffbookchapter,SayedProcIEEE}. The latter class has been shown in \cite{tu2012diffusion,sayed2013diffusion,SayedProcIEEE} to have superior mean-square-error and stability properties when constant step-sizes are used to enable continuous adaptation and learning, which is the main focus of this work. There are several variations of the diffusion strategy. It is sufficient to focus on the adapt-then-combine (ATC) version due to its enhanced performance; its update equations take the following form.
\begin{algorithm}[H]
\caption{Diffusion Strategy (ATC)} 
\label{alg:Diffusion}
\begin{subequations}
\begin{align}
	\boldsymbol{\psi}_{k,i} &= \boldsymbol{w}_{k,i-1} + \mu \u_{k,i}^\T (\d_k(i) - \u_{k,i} \boldsymbol{w}_{k,i-1}) \label{eq:diff_A}\\
	\boldsymbol{w}_{k,i} &= \sum_{\ell \in \mathcal{N}_k} a_{\ell k} \boldsymbol{\psi}_{\ell,i} \label{eq:diff_C}
\end{align}
\end{subequations}
\end{algorithm}
\noindent where $\mu > 0$ is a small step-size parameter and $\mathcal{N}_k$ denotes the neighborhood of agent $k$. Moreover, the coefficients $\{a_{\ell k}\}$ that comprise the $N \times N$ matrix $A$ are non-negative convex combination coefficients that satisfy the conditions:
\begin{align}
a_{\ell k}\geq0,\;\;\;\;
\sum_{\ell\in{\cal N}_{k}} a_{\ell k}=1,\;\;a_{\ell k}=0\;\mbox{\rm if }\ell\notin{\cal N}_k
\end{align}
In other words, the matrix $A$ is left-stochastic and satisfies $A^T\mathds{1}_N=\mathds{1}_N$. In \eqref{eq:diff_A}--\eqref{eq:diff_C}, each agent $k$ first updates its estimate $\boldsymbol{w}_{k,i-1}$ to an intermediate value by using its sensed data $\{\d_{k}(i),\u_{k,i}\}$ through \eqref{eq:diff_A}, and subsequently aggregates the information from the neighbors through \eqref{eq:diff_C}. In comparison, the update equations for the consensus strategy take the following form.
\begin{algorithm}[H]
\caption{Consensus Strategy}
\label{alg:consensus}
\begin{subequations}
\begin{align}
	\boldsymbol{\phi}_{k,i-1} &= \sum_{\ell \in \mathcal{N}_k} a_{\ell k} \boldsymbol{w}_{\ell,i-1} \label{eq:consensus_1}\\
	\boldsymbol{w}_{k,i} &= \boldsymbol{\phi}_{k,i-1} + \mu \u_{k,i}^\T (\d_k(i) - \u_{k,i} \boldsymbol{w}_{k,i-1}) \label{eq:consensus_2}
\end{align}
\end{subequations}
\end{algorithm}
\noindent It is important to note the asymmetry in the update \eqref{eq:consensus_2} with both $\{\boldsymbol{\phi}_{k,i-1},\boldsymbol{w}_{k,i-1}\}$ appearing on the right-hand side of \eqref{eq:consensus_2}, while the \emph{same} state variable $\boldsymbol{w}_{k,i-1}$ appears on the right-hand side of the diffusion strategy \eqref{eq:diff_A}. This asymmetry has been shown to be a  source for instability in consensus-based solutions \cite{tu2012diffusion,sayed2013diffusion,SayedProcIEEE}. 

A connected network is said to be strongly-connected when at least one $a_{kk}$ is strictly positive; i.e., there exists at least one agent with a self-loop, which is reasonable since it means that at least one agent in the network should have some trust in its own data. For such networks, when mean-square-error stability is ensured by using sufficiently small step-size parameters, the steady-state deviation (MSD) of the consensus and diffusion strategies can be shown to match to first-order in $\mu$ \cite{SayedProcIEEE}. Specifically, when $A$ is doubly-stochastic, it holds that: 
\begin{subequations}
\begin{align}
	\mathrm{MSD} 	 &\triangleq \lim_{i\rightarrow\infty} \frac{1}{N} \sum_{k=1}^N \E \|w^o - \boldsymbol{w}_{k,i}\|^2 \label{eq:network_MSD}\\
					 &= \frac{\mu}{2N} \Tr\left(\!\!\left(\sum_{k=1}^N R_{u,k}\right)^{\!\!-1} \!\!\left(\sum_{k=1}^N \sigma_{v,k}^2 R_{u,k}\right)\!\!\right) + O(\mu^{2}) \label{eq:diff_cons_MSD}
\end{align}
\end{subequations}
We emphasize the fact that the network need not be strongly-connected to ensure stability of the algorithm (convergence in the mean or mean-square error), but is only necessary to obtain an MSD expression of the form \eqref{eq:diff_cons_MSD}. In addition, more general expressions in the case where the matrix $A$ is not doubly-stochastic can be found in \cite{NOW_ML}. However, for the remainder of the manuscript, we will be focused on comparing the AH and AL algorithms to primal schemes with a doubly-stochastic combination matrix $A$.

It can be further shown that these strategies are able to equalize the MSD performance across the individual agents \cite{SayedProcIEEE,NOW_ML}. Specifically, if we let $\tilde{\boldsymbol{w}}_{k,i}=w^o-\boldsymbol{w}_{k,i}$, and define the individual MSD values as 
\begin{align}
\mathrm{MSD}_k \triangleq \lim_{i\rightarrow\infty} \E \|\tilde{\boldsymbol{w}}_{k,i}\|^2 
\end{align}
then it holds that $\mathrm{MSD}_k\;\doteq\;\mathrm{MSD}$ where the notation $a \doteq b$ means that the quantities $a$ and $b$ agree to first-order in $\mu$. This is a useful conclusion and it shows that the consensus and diffusion strategies are able to drive the estimates at the individual agents towards agreement within  $O(\mu)$ from the desired solution $w^o$ in the mean-square-error sense. Furthermore, it can be shown that the diffusion strategy \eqref{eq:diff_A}--\eqref{eq:diff_C} is guaranteed to converge in the mean for any connected network topology, as long as the agents are individually mean stable, i.e., whenever \cite{Cattivelli10}
 \begin{align}
 	0 < \mu < \min_{1\leq k\leq N} \left\{\frac{2}{\lambda_{\max}(R_{u,k})}\right\} \label{eq:diff_stability_range}
 \end{align} 
In contrast, consensus implementations can become unstable for some topologies even if all individual agents are mean stable \cite{tu2012diffusion,SayedProcIEEE}.

Relation \eqref{eq:diff_stability_range} also highlights a special feature of the diffusion strategies. We observe that the right-hand side of \eqref{eq:diff_stability_range} is independent of any information regarding the network topology over which the distributed algorithm is executed (it only requires that the network be connected). As we will see later in our presentation, this convenient attribute is not present in the distributed stochastic AH and AL algorithms studied. That is, the stability range of the AH and AL algorithms will explicitly depend on the network topology (see Section \ref{ssec:stepsize_range}). This is an inconvenient fact as it implies that the stability range for these algorithms may \emph{shrink} when more nodes are added or when nodes are added in a non-controlled way (see \eqref{eq:max_connectivity_large_eta_mu_range}).
% Furthermore, the algorithms do not require that the regressor covariance matrices \eqref{eq:R_u,k} are all positive-definite, but only require that the sum of them is positive-definite:
%\begin{align}
%	\sum_{k=1}^N R_{u,k} \succ 0, \quad R_{u,k} \nsucc 0, \quad k = 1,\ldots,N \label{eq:partial_observation}
%\end{align}
%This permits the case where none of the agents in the network not to be able to solve the problem individually, but instead \emph{must} cooperate in order to estimate $w^o$. In contrast, when each $R_{u,k} \succ 0$, then each agent can implement the non-cooperative LMS algorithm in order to estimate $w^o$ independently:
%\begin{algorithm}[H]
%\caption{No Cooperation}
%\label{alg:no_cooperation}
%\begin{align}
%	w_{k,i} &= w_{k,i-1} + \mu u_{k,i}^\T (d_k(i) - u_{k,i} w_{k,i-1}) \label{eq:no_coop}
%\end{align}
%\end{algorithm}

\section{Adaptive Primal-Dual Strategies}

There is an alternative method to encourage agreement among agents by explicitly introducing equality constraints into the problem formulation \eqref{eq:formulation_non_distributed_opt} --- see \eqref{eq:Opt_problem_global_constraints:Objective}--\eqref{eq:Opt_problem_global_constraints:Constraint} below. We are going to examine this alternative optimization problem and derive distributed variants for it. Such motivation of primal-dual algorithms has been previously utilized in \cite{schizas2009distributed,mateos2009performance,mateos2012distributed,Paolo}, without the use of the Laplacian and incidence matrices of the network topology as discussed in our exposition. Interestingly, it will turn out that while the primal-dual techniques studied herein generally perform well in the context of {\em deterministic} optimization problems, their performance is nevertheless degraded in nontrivial ways when approximation steps become necessary. One notable conclusion that follows from the subsequent analysis is that, strictly speaking,  there is no need to incorporate explicit equality constraints into the problem formulation as in \eqref{eq:Opt_problem_global_constraints:Constraint}. This is because this step ends up limiting the learning ability of the agents in comparison to the primal (consensus and diffusion) strategies. The analysis will clarify these statements. 

To motivate the adaptive primal dual strategy, we start by replacing \eqref{eq:formulation_non_distributed_opt} by the following equivalent  constrained optimization problem where the variable $w$ is replaced by $w_k$:
\begin{subequations}
\begin{align}
	\min_{\{w_k\}}\quad& \frac{1}{2} \sum_{k=1}^N  \E (\d_k(i) - \u_{k,i} w_k)^2	\label{eq:Opt_problem_global_constraints:Objective}\\
	\mathrm{s.t.}\quad& w_1 = w_2 = \cdots = w_N  	\label{eq:Opt_problem_global_constraints:Constraint}
\end{align}
\end{subequations}
The following definitions are useful \cite{bapat2010graphs,sayed2012diffbookchapter}. 
\begin{definition}[Incidence matrix of an undirected graph]
	\label{def:incidence}
	Given a graph $G$, the incidence matrix $C=[c_{ek}]$ is an $E \times N$ matrix, where $E$ is the total number of edges in the graph and $N$ is the total number of nodes, with entries defined as follows:
	\begin{align*}
		c_{e k} = \begin{mycases}
					+1, &k \mathrm{\;is\;the\;lower\;indexed\;node\;connected\;to\;} e\\
					-1, &k \mathrm{\;is\;the\;higher\;indexed\;node\;connected\;to\;} e\\
					 0, &\mathrm{otherwise}
				  \end{mycases}
	\end{align*}	
	 Thus, $C \mathds{1}_N = \mathbbold{0}_E$. Self-loops are excluded.\qed
\end{definition}
\begin{definition}[Laplacian matrix of a graph]
	Given a graph $G$, the Laplacian matrix $L=[l_{k\ell}]$ is an $N \times N$ matrix whose entries are defined as follows:
	\begin{align*}
		l_{k\ell} = \begin{mycases}
						|\mathcal{N}_k| - 1, & k = \ell\\
						-1, & k\neq \ell, \ell \in \mathcal{N}_k\\
						0, & \mathrm{otherwise}
					\end{mycases}
	\end{align*}
	where $|{\cal N}_k|$ denotes the number of neighbors of agent $k$. It holds that $L = C^\T C$, where $C$ is the incidence matrix.  \qed
\end{definition}
Since there must exist at least $N-1$ edges to connect $N$ nodes when the graph is connected, the number of edges in the graph satisfies $E \geq N-1$. Note that each node $k$ has access to the $k$-th row and column of the Laplacian matrix since they are aware of their local network connections. In addition, each node $k$ has access to row $e$ of the incidence matrix for which $c_{e k} \neq 0$. 

Since  the network is connected, it is possible to rewrite \eqref{eq:Opt_problem_global_constraints:Objective}-\eqref{eq:Opt_problem_global_constraints:Constraint} as
\begin{subequations}
\begin{align}
	\min_{\{w_k\}}\quad& \frac{1}{2} \sum_{k=1}^N \E (\d_k(i) - \u_{k,i} w_k)^2		\label{eq:Opt_problem_local_constraints:Objective}\\
	\mathrm{s.t.}\quad& \mathcal{C} \sw = \mathbbold{0}_{EM} 	\label{eq:Opt_problem_local_constraints:Constraint}
\end{align}
\end{subequations}
where we introduced the extended quantities:
\begin{align}
	\mathcal{C} \triangleq C \otimes I_M, \quad \sw \triangleq \mathrm{col}\{w_1,\ldots,w_N\}
\end{align}
The augmented Lagrangian of the constrained problem \eqref{eq:Opt_problem_local_constraints:Objective}--\eqref{eq:Opt_problem_local_constraints:Constraint} is given by \cite{cvx_book,poliak1987introduction}:
\begin{align}
	f(\sw,\lambda) &= \frac{1}{2} \sum_{k=1}^N \E (\d_k(i) - \u_{k,i} w_k)^2 + \lambda^\T \mathcal{C} \sw + \frac{\eta}{2} \|\sw\|^2_\mathcal{L} \label{eq:Problem_Formulation:Lagrangian_distributed}
\end{align}
where $\lambda \in \mathbb{R}^{E M \times 1}$ is the Lagrange multiplier vector: it consists of $E$ subvectors, $\lambda=\mbox{\rm col}\{\lambda_e\}$, each of size $M\times 1$ for $e=1,2,\ldots,E$. One subvector $\lambda_e$ is associated with each edge $e$. Moreover, 
\begin{align}
	\mathcal{L} \triangleq L \otimes I_M = \mathcal{C}^\T \mathcal{C} \label{eq:script_L}
\end{align}
and $\eta$ is a non-negative regularization parameter. The dual function $g(\lambda)$ associated with \eqref{eq:Problem_Formulation:Lagrangian_distributed} is found by minimizing $f(\sw,\lambda)$ over the primal variable $\sw$:
\begin{align}
	g(\lambda) = \min_\sw f(\sw,\lambda) \label{eq:dual_function}
\end{align}
The dual function $g(\lambda)$ is known to be always concave regardless of the convexity of the primal problem \cite[p.~216]{cvx_book}. An optimal dual variable $\lambda^o$ is found by maximizing \eqref{eq:dual_function} over $\lambda$:
\begin{align}
	\lambda^o = \underset{\lambda}{\arg\max}\ g(\lambda) \label{eq:dual_optimization}
\end{align}
The unique solution $\sw^o$ to \eqref{eq:Opt_problem_local_constraints:Objective}--\eqref{eq:Opt_problem_local_constraints:Constraint} can be determined by focusing instead on determining the saddle points $\{\sw^o,\lambda^o\}$ of the augmented Lagrangian function \eqref{eq:Problem_Formulation:Lagrangian_distributed} \cite{bertsekas1999nonlinear,cvx_book}. There  are various methods to do so, such as relying on a gradient ascent procedure when $g(\lambda)$ is known. This approach is the one taken in the Alternating Direction Method of Multipliers (ADMM) \cite{boyd2011distributed} and particularly in the work \cite[p.~356]{Paolo}. However, in the setting under study, the statistical moments of the data are not known and, therefore, the expectation in $f(\sw,\lambda)$ in \eqref{eq:Problem_Formulation:Lagrangian_distributed} cannot be evaluated beforehand. As a result, the dual function $g(\lambda)$ defined  by \eqref{eq:dual_function} cannot be determined either and, consequently, we cannot rely directly on \eqref{eq:dual_optimization} to determine the optimal dual variable. We will need to follow an alternative route that relies on the use of stochastic approximations. 

%\begin{table*}[!ht]
%	\renewcommand{\arraystretch}{2.2}
%	\caption{Comparison of the number of $M\times 1$ vector multiplications, additions, and exchanges per iteration at every node $k$. In the table, the symbol $n_k \triangleq |\mathcal{N}_k|$ denotes the number of neighbors of agent $k$. The number in brackets denotes the complexity that the lower-indexed node requires while the number outside parenthesis indicates the complexity at the higher-indexed agent.}
%	\label{Tab:Complexity}
%	\begin{center}
%	\begin{tabular}{c||c|c|c}
%	\hline \hline
%	\rowcolor[gray]{0.9}{\normalsize Algorithm }& {\normalsize Multiplications} & {\normalsize Additions} & {\normalsize Exchanges}\\ 
%	\hline 
%	Diffusion Strategy \eqref{eq:diff_A}--\eqref{eq:diff_C} & $n_k + 1$ & $n_k$ & $n_k-1$\\ 
%	\hline 
%	Consensus Strategy \eqref{eq:consensus_1}--\eqref{eq:consensus_2} & $n_k + 1$ & $n_k$ & $n_k-1$\\ 
%	\hline 
%	Arrow-Hurwicz Method \eqref{eq:AH_primal_alg}--\eqref{eq:AH_dual_alg} & $n_k + 2 \ [n_k + 3]$ & $2n_k\ [2n_k + 2]$ & $2 (n_k - 1)\ [n_k-1]$\\ 
%	\hline 
%	Augmented Lagrangian Method \eqref{eq:AL_primal_alg}--\eqref{eq:AL_dual_alg} & $n_k + 4 \ [n_k + 5]$ & $3 n_k \ [3 n_k + 2]$ & $2 (n_k - 1)$ \\ 
%	\hline \hline
%	\end{tabular} 
%	\end{center}
%\end{table*}

We search for a saddle-point of \eqref{eq:Problem_Formulation:Lagrangian_distributed} by employing a {\em stochastic approximation} version of the first-order AL algorithm \cite[pp.~240--242]{poliak1987introduction} \cite[p.~456]{bertsekas1999nonlinear}. The implementation relies on a {\em stochastic} gradient descent step with respect to the primal variable, $\sw$, and a gradient ascent step with respect to the dual variable, $\lambda$, as follows:
\begin{subequations}
\begin{align}
	\w_i &= \w_{i-1} - \mu \widehat{\nabla_w} f(\w_{i-1}, \boldsymbol{\lambda}_{i-1}) \label{eq:AL_primal}\\
	\boldsymbol{\lambda}_i &= \boldsymbol{\lambda}_{i-1} + \mu \nabla_\lambda f(\w_{i-1}, \boldsymbol{\lambda}_{i-1}) \label{eq:AL_dual}
\end{align}
\end{subequations}
Observe that we are using an approximate gradient vector in \eqref{eq:AL_primal} and the exact gradient vector in \eqref{eq:AL_dual}; this is because differentiation relative to $\sw$ requires knowledge of the data statistics, which are not available. These gradient vectors are evaluated as follows: 
\begin{subequations}
\begin{align}
	\widehat{\nabla_{w}} f(\w,\boldsymbol{\lambda}) &= \h_i + \mathcal{C}^\T \boldsymbol{\lambda} + \eta \mathcal{L} \w \label{eq:gradient_lagrangian_w}\\
	\nabla_\lambda f(\w,\boldsymbol{\lambda}) &= \mathcal{C} \w \label{eq:gradient_lagrangian_lambda}
\end{align}
\end{subequations}
where the vector $\h_i$ amounts to an instantaneous approximation for the gradient vector of the first term on the right-hand side of \eqref{eq:Problem_Formulation:Lagrangian_distributed}; its $k-$th entry is given by $-\u_{k,i}^{\T}(\d_k(i)-\u_{k,i}\boldsymbol{w}_k)$. Substituting into \eqref{eq:AL_primal}--\eqref{eq:AL_dual}, we obtain the following algorithm, where the notation $\bm{\lambda}_{e,i}$ denotes the estimate for the subvector $\lambda_e$ associated with the edge of index $e$. 
\begin{algorithm}[H]
	\caption{Distributed Augmented Lagrangian (AL)}
	\label{alg:AL}
	{\small \begin{subequations}
\begin{align}
	\boldsymbol{\psi}_{k,i-1} &= \boldsymbol{w}_{k,i-1} \!-\! \mu \sum_{e= 1}^E c_{e k} \boldsymbol{\lambda}_{e,i-1} \!-\! \mu  \eta \sum_{\ell \in \mathcal{N}_k} l_{k \ell} \boldsymbol{w}_{\ell,i-1} \label{eq:AL_primal_alg}\\
	\boldsymbol{w}_{k,i} &= \boldsymbol{\psi}_{k,i-1} \!+\! \mu \u_{k,i}^\T (\d_{k}(i) \!-\! \u_{k,i} \boldsymbol{w}_{k,i-1})  \label{eq:AL_primal_alg2}\\
	\boldsymbol{\lambda}_{e,i} &= \boldsymbol{\lambda}_{e,i-1} + \mu (\boldsymbol{w}_{k,i-1} - \boldsymbol{w}_{\ell,i-1}) \quad [\ell \!>\! k, \ell \!\in\! \mathcal{N}_k]\label{eq:AL_dual_alg}
\end{align}
\end{subequations}}
\end{algorithm}
\noindent When $\eta = 0$ in \eqref{eq:AL_primal_alg}, we obtain the \emph{distributed AH} method, also considered in \cite{Ribeiro,Sergio} for the solution of saddle point problems for other cost functions. Reference \cite{Sergio} considers problems that arise in the context of reinforcement learning in response to target policies, while reference \cite{Ribeiro} considers regret analysis problems and employs {\em decaying} step-sizes rather than continuous adaptation. As noted earlier, constant step-sizes enrich the dynamics in non-trivial ways due to the persistent presence of gradient noise. Note further that by setting $\eta=0$, step \eqref{eq:AL_primal_alg} will end up relying solely on the dual variables and will not benefit from the neighbors' iterates. The presence of these iterates in \eqref{eq:AL_primal_alg} strongly couples the dynamics of the various steps in the distributed Lagrangian implementation. 

In the above statement, either node connected to edge $e$ can be responsible for updating $\boldsymbol{\lambda}_{e,i-1}$. At each step, the nodes communicate their previous variables $(\boldsymbol{w}_{k,i-1}, \{\boldsymbol{\lambda}_{e,i-1}\})$ to their neighbors and then each node executes the steps outlined in Alg.~\ref{alg:AL}. The primal-dual algorithms require more computation and communication per iterations than the primal methods (generally, by at least a factor of $2$; this is because the agents in the AL and AH implementations also need to propagate the additional dual variables, $\bm{\lambda}_{e,i-1}$).

\begin{table*}[!ht]
	\renewcommand{\arraystretch}{2.2}
	\centering
	\begin{threeparttable}
	\caption{Summary of Results: Partial Observation implies that when some of the covariance matrices are singular but their sum is positive-definite, the algorithm can estimate $w^o$. A stability range that depends on the Laplacian matrix $L$ implies that the stability of the algorithm depends on the network topology.}
	\label{Tab:PreviewOfResults}
	\begin{tabular}{c||c|c|c}
	\hline \hline
	\rowcolor[gray]{0.9}{ Algorithm }& { Handles Partial Observation} & { Stability Range} & { Steady-state MSD} \\ 
	\hline 
	Diffusion Strategy \eqref{eq:diff_A}--\eqref{eq:diff_C} & \cmark & $0 < \mu < \bar{\mu}$ & $\frac{1}{N} \mathrm{MSD}^{NC}$ \\ 
	\hline 
	Consensus Strategy \eqref{eq:consensus_1}--\eqref{eq:consensus_2} & \cmark & $0 < \mu < \mu^c(L) < \bar{\mu}$ & $\frac{1}{N} \mathrm{MSD}^{NC}$ \\ 
	\hline 
	No Cooperation & \xmark & $0 < \mu < \bar{\mu}$ & $\mathrm{MSD}^{NC} \triangleq \frac{\mu M}{2 N} \displaystyle \sum_{k=1}^N  \sigma_{v,k}^2$ \\ 
	\hline 
	Arrow-Hurwicz (AH) Method \eqref{eq:AL_primal_alg}--\eqref{eq:AL_dual_alg} with $\eta = 0$ & \xmark \tnote{a} & $0 < \mu < \mu^{AH}(L) < \bar{\mu}$ & $\mathrm{MSD}^{NC}$ \\ 
	\hline 
	Augmented Lagrangian (AL) Method \eqref{eq:AL_primal_alg}--\eqref{eq:AL_dual_alg} for large $\eta$ & \cmark \tnote{b} & $0 < \mu < \frac{\mu^{AL}(L)}{\eta}$ & $\frac{1}{N} \mathrm{MSD}^{NC} + O\left(\frac{1}{\eta}\right)$ \\ 
	\hline \hline
	\end{tabular} 
	\begin{tablenotes}
		\item[a] See Corollary \ref{cor:generalTopology_mean_stability}.
		\item[b] See Theorems \ref{thm:generalTopology_mean_stability_partial_observability} and \ref{thm:generalTopology_mean_square_stability_partial_observability}.
	\end{tablenotes}
	\end{threeparttable}
\end{table*}

\section{Preview of Results}
We summarize in this section the main highlights to be derived in the sequel. 

To begin with, it is known that even when some of the regression covariance matrices, $R_{u,k}$, are singular, the diffusion strategy \eqref{eq:diff_A}--\eqref{eq:diff_C} and the consensus strategy \eqref{eq:consensus_1}--\eqref{eq:consensus_2} are still able to estimate $w^o$ through the collaborative process among the agents \cite{SayedProcIEEE}. We will verify in the sequel that the AL algorithm \eqref{eq:AL_primal_alg}--\eqref{eq:AL_dual_alg} can also converge in this case (see Theorems \ref{thm:generalTopology_mean_stability_partial_observability} and \ref{thm:generalTopology_mean_square_stability_partial_observability}), while the AH algorithm need not converge.  Further differences among the algorithms arise in relation to how stable they are and how close their iterates get to the desired $w^o$ when they do converge. We will see, for instance, that the AH and AL implementations are stable over a smaller range of step-sizes than consensus and diffusion (and, therefore, are ``less'' stable). We will also see that even when they are stable, the MSD performance of the primal-dual networks (networks that utilize the AH and AL algorithms) is degraded relative to what is delivered by consensus and diffusion strategies. 

We will also examine the useful case when all regression covariance matrices are uniform and positive-definite:
\begin{align}
	R_{u,k} = R_u > 0,\quad\quad k=1,\ldots,N \label{eq:common_Ru}
\end{align} 
Under \eqref{eq:common_Ru}, we will show that all algorithms under study can recover $w^o$, but that the step-size range for stability for the AH and AL algorithms continues to be smaller even in comparison to the non-cooperative solution where nodes act independently of each other. In other words, cooperation does not necessarily enhance the stability range of primal-dual networks (defined as the step-size range that allows the algorithm to at least converge in the mean). In contrast, diffusion adaptation always enhances the stability range and leads to more relaxed stability conditions than non-cooperative processing \cite{SayedProcIEEE,NOW_ML} (such as not requiring that all $\{R_{u,k}\}$ are positive-definite). In addition, we will show that the step-size range for primal-dual networks depends on their topology and is inversely proportional to $\eta$ in the case of the AL algorithm (see Theorem \ref{thm:eigenvalues_minusR2}). What these results mean is that connected nodes can fail to converge via the AH and AL algorithms, even though they will converge under the diffusion and consensus strategies \eqref{eq:diff_A}--\eqref{eq:diff_C}, \eqref{eq:consensus_1}--\eqref{eq:consensus_2}, or even under the non-cooperative solution (see Example \ref{ex:diverging_example}). A similar observation was proven earlier in \cite{tu2012diffusion} for consensus strategies: the network can become unstable even  if all individual agents are stable. The reason for this behavior is the asymmetry noted earlier in the consensus update \eqref{eq:consensus_2}; this asymmetry can lead to an unstable growth in the state of consensus networks. We observe from \eqref{eq:AL_primal_alg2} that a similar asymmetry exists in the update equations for AL and AH implementations. Diffusion strategies, on the other hand, do not have this asymmetry and will remain stable regardless of the topology \cite{tu2012diffusion,SayedProcIEEE}.

With regards to the MSD performance \eqref{eq:network_MSD}, we will show that the AH algorithm achieves the same performance level as the non-cooperative algorithm (see Corollary \ref{cor:zero_eta_approximation}), while the AL algorithm with $\eta > 0$ attains a performance level that is worse than that by the diffusion and consensus strategies by a positive additive term that decays as $\eta$ increases. The above observations and results are summarized in Table \ref{Tab:PreviewOfResults}. The stability ranges and steady-state MSD values are under \eqref{eq:common_Ru} and assume that the Metropolis combination weights \cite{sayed2012diffbookchapter} are used for diffusion and consensus.

\section{Coupled Error Recursion}
We move on to establish the above results by carrying out a detailed mean-square-error analysis of the algorithms. 
\subsection{Error Quantities}
We know that the optimizer of \eqref{eq:Opt_problem_local_constraints:Objective}--\eqref{eq:Opt_problem_local_constraints:Constraint} is $w_k = w^o$ for all $k=1,\ldots,N$, where $w^o$ was defined in \eqref{eq:ProblemFormulation:LinearModel} since \eqref{eq:Opt_problem_local_constraints:Objective}--\eqref{eq:Opt_problem_local_constraints:Constraint} is equivalent to \eqref{eq:formulation_non_distributed_opt}. We introduce the error vector at each agent $k$, $\widetilde{\boldsymbol{w}}_{k,i} = w^o-\boldsymbol{w}_{k,i}$, and collect all errors from across the network into the block column vector:
\begin{align}
\widetilde{\w}_i=\mbox{\rm col}\{\widetilde{\boldsymbol{w}}_{1,i},\,\widetilde{\boldsymbol{w}}_{2,i},\ldots,\widetilde{\boldsymbol{w}}_{N,i}\}
\end{align}
We next subtract $w^o$ from both  sides of \eqref{eq:AL_primal_alg} and use \eqref{eq:ProblemFormulation:LinearModel} to find that the network error vector evolves according to the following dynamics:
\begin{equation}
	\widetilde{\w}_{i} \!=\! \widetilde{\w}_{i-1} \!+\! \mu \left(-\boldsymbol{\mathcal{H}}_i \widetilde{\w}_{i-1} \!+\! \mathcal{C}^\T \boldsymbol{\lambda}_{i-1} \!+\! \eta \mathcal{L} \w_{i-1}\right) \!-\! \mu \z_i \label{eq:ErrorRecursion:primal_error_recursion1}
\end{equation}
where
\begin{align}
	\z_i &\triangleq \mathrm{col}\{\u_{1,i} \v_1(i),\ldots \u_{N,i} \v_N(i)\} \label{eq:z_i}\\
	\boldsymbol{\mathcal{H}}_i &\triangleq \textrm{blockdiag}\{\u_{1,i}^\T \u_{1,i},\ldots,\u_{N,i}^\T \u_{N,i}\} \label{eq:H_i}
\end{align}

\noindent We know that there exists a vector $\lambda^o$ (possibly not unique) that satisfies \cite[Ch.~5]{cvx_book}
\begin{align}
\nabla_\sw f(\mathds{1}_N \otimes w^o,\lambda^o) &= \mathbbold{0}_{NM}
\end{align}
But since $\mathcal{L} \mathds{1}_{N M} = \mathbbold{0}_{N M}$, this condition implies that 
\begin{align}
	\underset{1\leq k\leq N}{\mathrm{col}} \left\{R_{u,k}w^o - r_{du,k}\right\} + \mathcal{C}^\T \lambda^o = \mathbbold{0}_{NM} \label{eq:ErrorRecursion:KKT}
\end{align}
Moreover, since $w^o$ optimizes \eqref{eq:formulation_non_distributed_opt}, we have that $w^o$ satisfies:
\begin{align}
	\underset{1\leq k\leq N}{\mathrm{col}} \left\{R_{u,k}w^o - r_{du,k}\right\} = \mathbbold{0}_{NM}
\end{align}
and we conclude from \eqref{eq:ErrorRecursion:KKT} that we must have:
\begin{align}
	\mathcal{C}^\T \lambda^o = \mathbbold{0}_{NM} \label{eq:condition_optimal_lambda}
\end{align}

\subsection{Useful Eigen-Spaces} 
The rank-deficiency of $\mathcal{C}$ creates difficulties for the study of the stability and performance
of the adaptive primal-dual networks. We will resort to a useful transformation that allows us to identify and ignore redundant dual variables. We start by introducing the singular-value-decomposition of $C$:
\begin{align}
	C = U S V^\T \label{eq:SVD_C}
\end{align}
where $U \in \mathbb{R}^{E\times E}$ and $V \in \mathbb{R}^{N\times N}$ are orthogonal matrices and $S \in \mathbb{R}^{E\times N}$ is partitioned according to
\begin{align}
	S = \left[\begin{array}{c|c}
	S_2 & \mathbbold{0}_{N-1}\\\hline
	0_{(E-N+1) \times (N-1)} & \mathbbold{0}_{E-N+1}
	\end{array}\right]
\end{align}
where the square diagonal matrix $S_2 \in \mathbb{R}^{(N-1)\times (N-1)}$ contains the nonzero singular values of $C$ along its main diagonal and is therefore non-singular.

Now, since $L=C^\T C$, it follows that $S_2^\T S_2$ is a diagonal matrix containing the nonzero eigenvalues of $L$, i.e., $L = V D V^\T$, where
\begin{align}
	D \triangleq \left[\begin{array}{cc} 
								D_1 & \mathbbold{0}_{N-1}\\
								\mathbbold{0}_{N-1}^\T & 0
							\end{array}
						\right] = \left[\begin{array}{cc} 
								S_2^\T S_2 & \mathbbold{0}_{N-1}\\
								\mathbbold{0}_{N-1}^\T & 0
							\end{array}
						\right] \label{eq:Lambda}
\end{align}
in terms of $D_1=S_2^\T S_2$. Furthermore, since $L \mathds{1}_N = \mathbbold{0}_N$, we can partition $V$ into:
\begin{align}
V = \left[\begin{array}{cc}  V_2 & \frac{1}{\sqrt{N}} \mathds{1}_N\end{array}\right] \label{eq:subdivided_V}
\end{align}
In addition, we can write 
\begin{align}
	\mathcal{C} = \mathcal{U} \mathcal{S} \mathcal{V}^\T \label{eq:SVD_script_C}
\end{align}
where ${\cal U} = U \otimes I_M$, ${\cal V} = V \otimes I_M$, and ${\cal S} = S \otimes I_M$ can be partitioned as:
\begin{align}
{\cal S} &= \left[\begin{array}{c|c}
	{\cal S}_2 & 0_{(N-1)M \times M}\\\hline
	0_{(E-N+1)M\times N M} & 0_{(E-N+1)M \times M}
	\end{array}\right] \label{eq:subdivided_S}
\end{align}
with a nonsingular diagonal matrix ${\cal S}_2 \in \mathbb{R}^{(N-1)M \times (N-1)M}$. Likewise, 
\begin{align}
	\mathcal{V} &= \left[\begin{array}{cc} \mathcal{V}_2 & \mathcal{V}_0 \end{array}\right] \label{eq:subdivided_script_V}
\end{align}
where $\mathcal{V}_2 \triangleq V_2 \otimes I_M$ and $\mathcal{V}_0 \triangleq \mathds{1}_N/\sqrt{N} \otimes I_M$. 

\subsection{Dimensionality Reduction}
We recall \eqref{eq:AL_dual} and re-write it as 
\begin{align}
	\boldsymbol{\lambda}_i &= \boldsymbol{\lambda}_{i-1} - \mu \mathcal{C} (\mathds{1}_N \otimes w^o - \w_{i-1}) \label{eq:alg_recursion1}
\end{align}
Next, we introduce the transformed vectors:
\begin{align}
	{\lambda'}^o = \mathcal{U}^\T \lambda^o,\quad\quad {\sw'}^o = \mathcal{V}^\T (\mathds{1}_N \otimes w^o)
\end{align}
where we are using the prime notation to refer to transformed quantities. Then, relation \eqref{eq:condition_optimal_lambda} implies that 
\begin{align}
\mathcal{S}^\T {\lambda'}^o = 0	 \label{eq:optimality_condition_joint}
\end{align}
We partition ${\lambda'}^o$ as:
\begin{align}
{\lambda'}^o = \left[\begin{array}{c}
	{\lambda'_1}^{\!o} \\ \hline
	{\lambda'_2}^{\!o}
	\end{array} \right]
\end{align}
where ${\lambda'}_1^o \in \mathbb{R}^{(N-1)M\times 1}$ is the dual variable associated with $N-1$ constraints in the network and ${\lambda'}_2^o \in \mathbb{R}^{(E-N+1)M\times 1}$ are the dual variables associated with the remaining constraints. We then conclude from \eqref{eq:optimality_condition_joint} that 
\begin{align}
	\mathcal{S}_2^\T {\lambda'_1}^{\!o} = \mathbbold{0}_{(N-1)M} \label{eq:optimality_condition_split}
\end{align}
Observe that while the optimal Lagrange multiplier $\lambda^o$ may not be unique, the transformed vector ${\lambda'}_1^o$ is unique since $\mathcal{S}_2$ is invertible. We multiply both sides of \eqref{eq:alg_recursion1} from the left by $\mathcal{U}^\T$ to obtain
\begin{align}
	\boldsymbol{\lambda}_i' &= \boldsymbol{\lambda}_{i-1}' - \mu \mathcal{S} ({\sw'}^o - \w_{i-1}') \label{eq:transformed_alg_recursion}
\end{align}
where 
\begin{align}
	\boldsymbol{\lambda}_i' \triangleq \mathcal{U}^\T\boldsymbol{\lambda}_i ,\quad\quad \w_i' \triangleq \mathcal{V}^\T \w_i \label{eq:transformed_variables}
\end{align}
We similarly partition ${\sw'}^o$, $\w_i'$, and $\boldsymbol{\lambda}_i'$:
\begin{align}
	\boldsymbol{\lambda}_i' = \left[\begin{array}{c}
	\boldsymbol{\lambda}_{1,i}' \\ \hline
	\boldsymbol{\lambda}_{2,i}'
	\end{array} \right], \ \ 
	{\sw'}^o = \left[\begin{array}{c}
	{\sw'}^o_{1} \\ \hline
	{\sw'}^o_{2}
	\end{array} \right]	, \ \  \w_i' = \left[\begin{array}{c}
	\w_{1,i}' \\ \hline
	\w_{2,i}'
	\end{array} \right] \label{eq:partitioning_w}
\end{align}
where ${\sw'}^o_{1},\w_{1,i}' \in \mathbb{R}^{(N-1)M\times 1}$, ${\sw'}^o_{2},\w_{2,i}' \in \mathbb{R}^{M\times 1}$, $\boldsymbol{\lambda}_{1,i}' \in \mathbb{R}^{(N-1)M\times 1}$, and $\boldsymbol{\lambda}_{2,i}' \in \mathbb{R}^{(E-N+1)M\times 1}$. Rewriting \eqref{eq:transformed_alg_recursion} in terms of \eqref{eq:partitioning_w}, we obtain
\begin{align}
	\left[\begin{array}{c}
	\boldsymbol{\lambda}_{1,i}' \\ \hline
	\boldsymbol{\lambda}_{2,i}'
	\end{array} \right] &= \left[\begin{array}{c}
	\boldsymbol{\lambda}_{1,i-1}' \\ \hline
	\boldsymbol{\lambda}_{2,i-1}'
	\end{array} \right] - \mu \left[\begin{array}{c}{\cal S}_2 ({\sw_1'}^o - \w_{1,i-1}')\\\hline 0_{(E-N+1)M\times N M}\end{array}\right]  \label{eq:transformed_alg_recursion_split}
\end{align}
We observe from \eqref{eq:transformed_alg_recursion_split} that $\boldsymbol{\lambda}_{2,i-1}'$ does not change as the algorithm progresses. It is therefore  sufficient to study the evolution of $\boldsymbol{\lambda}_{1,i}'$ alone:
\begin{align}
	\boldsymbol{\lambda}_{1,i}' &= \boldsymbol{\lambda}_{1,i-1}' - \mu \mathcal{S}_2 ({\sw_1'}^o - \w_{1,i-1}') \label{eq:transformed_alg_recursion_reduced}
\end{align} 
We may now subtract \eqref{eq:transformed_alg_recursion_reduced} from ${\lambda_1'}^{o}$ to obtain the error-recursion:
\begin{align}
	\widetilde{\boldsymbol{\lambda}}_{1,i}' &= \widetilde{\boldsymbol{\lambda}}_{1,i-1}' + \mu \mathcal{S}_2 \widetilde{\w}_{1,i-1}' \label{eq:transformed_alg_recursion_reduced2}
\end{align} 
where
\begin{align}
	\widetilde{\boldsymbol{\lambda}}_{1,i}' \triangleq {\lambda_1'}^o - \boldsymbol{\lambda}_{1,i}', \quad\quad \widetilde{\w}_{1,i}' \triangleq {\sw'}^o_1 - \w_{1,i}' \label{eq:transformed_error_quantities}
\end{align}
On the other hand, since $\mathcal{L} (\mathds{1}_N \otimes w^o) = \mathbbold{0}_{NM}$, we have that \eqref{eq:ErrorRecursion:primal_error_recursion1} can be re-written as
\begin{align}
	\widetilde{\w}_{i} = \widetilde{\w}_{i-1} \!+\! \mu \left(-\boldsymbol{\mathcal{H}}_i \mathcal{V} \widetilde{\w}'_{i-1} \!+\! \mathcal{C}^\T \boldsymbol{\lambda}_{i-1} \!-\! \eta \mathcal{L} \mathcal{V} \widetilde{\w}'_{i-1}\right) - \mu \z_i
\end{align}
Multiplying both sides from the left side by $\mathcal{V}^\T$, we obtain
\begin{align}
	\widetilde{\w}_{i}' &= \widetilde{\w}_{i-1}' \!-\! \mu \left((\mathcal{V}^\T \boldsymbol{\mathcal{H}}_i \mathcal{V} \!+\! \eta \sLambda)\widetilde{\w}'_{i-1} \!-\! \mathcal{V}^\T \mathcal{C}^\T 
\boldsymbol{\lambda}_{i-1} \right) - \mu \z_i'  \label{eq:ErrorRecursion:primal_error_recursion}
\end{align}
where 
\begin{align}
	\sLambda \triangleq D \otimes I_M = \left[\begin{array}{c|c}
	\mathcal{S}_2^\T \mathcal{S}_2 & 0_{(N-1)M\times M} \\ \hline
	0_{M\times (N-1)M} & 0_{M\times M}
	\end{array} \right] \label{eq:sLambda}
\end{align}
and
\begin{align}
	\z_i' \triangleq \mathcal{V}^\T \z_i = \left[\begin{array}{c}
	\mathcal{V}_2^\T \z_i \\ \hline
	\mathcal{V}_0^\T \z_i
	\end{array}\right] \label{eq:transformed_z}
\end{align}
Using \eqref{eq:Lambda} and \eqref{eq:subdivided_script_V} we have:
\begin{align}
	\mathcal{V}^\T \boldsymbol{\mathcal{H}}_i \mathcal{V}  + \eta \sLambda &= \left[\begin{array}{cc}
	\mathcal{V}_2^\T \boldsymbol{\mathcal{H}}_i \mathcal{V}_2 + \eta \mathcal{S}_2^\T \mathcal{S}_2& \mathcal{V}_2^\T \boldsymbol{\mathcal{H}}_i \mathcal{V}_0 \\ 
	\mathcal{V}_0^\T \boldsymbol{\mathcal{H}}_i \mathcal{V}_2 & \mathcal{V}_0^\T \boldsymbol{\mathcal{H}}_i \mathcal{V}_0
	\end{array} \right] \label{eq:V'HV+e*L}
\end{align}
and using \eqref{eq:subdivided_S} we also have
\begin{align}
	\mathcal{V}^\T \mathcal{C}^\T \boldsymbol{\lambda}_{i-1} = \mathcal{S}^\T \boldsymbol{\lambda}_{i-1}' 
	\stackrel{(a)}{=} \left[\begin{array}{c}
	-\mathcal{S}_2^\T \\ 
	0_{M \times (N-1)M}
	\end{array} \right] \widetilde{\boldsymbol{\lambda}}_{1,i-1}' \label{eq:dual_term}
\end{align}
where step $(a)$ is due to \eqref{eq:optimality_condition_split}.

Collecting \eqref{eq:transformed_alg_recursion_reduced2}, \eqref{eq:ErrorRecursion:primal_error_recursion}, and \eqref{eq:V'HV+e*L}--\eqref{eq:dual_term} in matrix form, we arrive at the following theorem for the evolution of the error dynamics of the primal-dual strategy over time.

\begin{theorem}[Error dynamics of primal-dual strategies]
Let the network be connected. Then, the error dynamics of the primal-dual algorithms evolves over time as follows:
\begin{align}
	\left[\begin{array}{c}
	\widetilde{\w}_{1,i}' \\ 
	\widetilde{\w}_{2,i}' \\ 
	\widetilde{\boldsymbol{\lambda}}_{1,i}'
	\end{array}\right]  = \boldsymbol{\mathcal{B}}_i'  \left[\begin{array}{c}
	\widetilde{\w}_{1,i-1}' \\ 
	\widetilde{\w}_{2,i-1}' \\ 
	\widetilde{\boldsymbol{\lambda}}_{1,i-1}'
	\end{array}\right] - \mu \left[\begin{array}{c}
	\mathcal{V}_2^\T\z_i \\ 
	\mathcal{V}_0^\T\z_i \\ 
	\mathbbold{0}_{(N-1)M}
	\end{array} \right] \label{eq:matrix_form_random_error_recursion}
\end{align}
where 
\begin{align}
	\boldsymbol{\mathcal{B}}_i' &\triangleq I_{(2N-1)M} - \mu \boldsymbol{\mathcal{R}}_i'\\
	\boldsymbol{\mathcal{R}}_i' &\triangleq \left[\!\!\begin{array}{cc|c}
	\mathcal{V}_2^\T \boldsymbol{\mathcal{H}}_i \mathcal{V}_2 + \eta \mathcal{S}_2^\T \mathcal{S}_2& \mathcal{V}_2^\T \boldsymbol{\mathcal{H}}_i \mathcal{V}_0 & \mathcal{S}_2^\T\\
	\mathcal{V}_0^\T \boldsymbol{\mathcal{H}}_i \mathcal{V}_2 & \mathcal{V}_0^\T \boldsymbol{\mathcal{H}}_i \mathcal{V}_0 &  0_{M \times (N-1)M} \\ \hline
	-\mathcal{S}_2 & 0_{(N-1)M\times M} & 0_{(N-1)M}
	\end{array}\!\!\right]
\end{align}
and $\widetilde{\w}_{1,i}'$ and $\widetilde{\boldsymbol{\lambda}}_{1,i}'$ are defined in \eqref{eq:transformed_error_quantities}, $\widetilde{\w}_{2,i}' \triangleq {\sw'_2}^o - \w_{2,i}'$, and $\boldsymbol{\mathcal{H}}_i$ is defined in \eqref{eq:H_i}.
\qed
\end{theorem}
It is clear from \eqref{eq:transformed_variables} that if $\E \widetilde{\w}_i' \triangleq \E (\sw'^o - \w_i')$ converges to zero, then $\E \widetilde{\w}_i\triangleq \E(\sw^o - \w_i)$ also converges to zero since $\widetilde{\w}_i' = \mathcal{V}^\T \widetilde{\w}_i$. Furthermore, it holds that, for any $NM\times NM$ positive-semidefinite real matrix $\Sigma$, 
\begin{align}
	\E\|\widetilde{\w}_{i}\|^2_\Sigma = \E \|\widetilde{\w}_{i}'\|^2_{\Sigma'} \label{eq:MSD_equivalence}
\end{align}
where $\Sigma' \triangleq \mathcal{V}^\T \Sigma \mathcal{V}$. Therefore, for mean-square-error analysis, it is sufficient to examine the behavior of $\E \|\widetilde{\w}_{i}'\|^2_{\Sigma'}$, knowing that we can relate it back to $\E\|\widetilde{\w}_{i}\|^2_\Sigma$ via \eqref{eq:MSD_equivalence}.

\subsection{Mean Error Recursion}
To obtain the mean error recursion, we compute expectations of both sides of \eqref{eq:matrix_form_random_error_recursion} and use that $\E \z_i=\mathbbold{0}_{NM}$  to get 
\begin{align}
	\left[\begin{array}{c}
	\E \widetilde{\w}_{1,i}' \\ 
	\E \widetilde{\w}_{2,i}' \\ 
	\E \widetilde{\boldsymbol{\lambda}}_{1,i}'
	\end{array}\right]  = \mathcal{B}' \left[\begin{array}{c}
	\E \widetilde{\w}_{1,i-1}' \\ 
	\E \widetilde{\w}_{2,i-1}' \\ 
	\E \widetilde{\boldsymbol{\lambda}}_{1,i-1}'
	\end{array}\right] \label{eq:first_mean_recursion}
\end{align}
where
\begin{align}
	\mathcal{B}' \!&\triangleq \!\E \boldsymbol{\mathcal{B}}_i' = I_{(2N-1)M} - \mu \mathcal{R}' \label{eq:B_prime}\\
	\mathcal{R}' \!&\triangleq \!\left[\!\!\begin{array}{cc|c} \mathcal{V}_2^\T \mathcal{H} \mathcal{V}_2 + \eta \mathcal{S}_2^\T \mathcal{S}_2 & \mathcal{V}_2^\T \mathcal{H} \mathcal{V}_0 & \mathcal{S}_2^\T\\
													  \mathcal{V}_0^\T \mathcal{H} \mathcal{V}_2 & \displaystyle \frac{1}{N} \sum_{k=1}^N R_{u,k} & 0_{M\!\times\!(N-1)M}\\ \hline
													   -\mathcal{S}_2 & 0_{(N-1)M\!\times\!M} & 0_{(N-1)M}  \end{array}\!\!\right]  \label{eq:R_prime}\\
	%\mathcal{K} &\triangleq \mathcal{V}^\T \mathcal{H} \mathcal{V} + \eta \mathcal{S}_1^\T \mathcal{S}_1 \label{eq:script_k}\\
	\mathcal{H} \!&\triangleq \!\mathrm{blockdiag}\{R_{u,1},\ldots,R_{u,N}\}
\end{align}
We would like to determine conditions on the step-size and regularization parameters $\{\mu,\eta\}$ to ensure asymptotic convergence of the mean quantity $\E \widetilde{\w}'_i$ to zero. We will examine this question later in Sec.~\ref{sec:Stability} when we study the stability of ${\cal B}'$. In the next section, we derive the mean-square-error recursion.\\

\noindent {\bf Remark} (Interpretation of transformed variables). We already observed that the transformed dual variables $\lambda_1'$ correspond to ``useful'' constraints while the dual variables $\lambda_2'$ correspond to ``redundant'' constraints. We can also obtain an interesting interpretation for the transformed variable $\sw'$ in the special case where $R_{u,1} = \ldots = R_{u,N} = R_u$ (i.e., $\mathcal{H} = I_N \otimes R_u$). First, partition the variable $\sw'$ according to
\begin{align}
	\w'_i = \left[\begin{array}{c} \w_{1,i}'\\ \boldsymbol{w}_{2,i}' \end{array}\right] \label{eq:partitioned_w}
\end{align}
where $\w_{1,i}' \in \mathbb{R}^{(N-1)M\times 1}$ and $\boldsymbol{w}_{2,i}' \in \mathbb{R}^{M\times 1}$. Substituting  \eqref{eq:subdivided_V}, and \eqref{eq:partitioned_w} into \eqref{eq:first_mean_recursion}, we obtain the following recursion:
\begin{align}
	\left[\begin{array}{c}
	\E \widetilde{\w}_{1,i}' \\ 
	\E \widetilde{\boldsymbol{w}}_{2,i}' \\ 
	\E \widetilde{\boldsymbol{\lambda}}_{1,i}'
	\end{array}\right]  &=  \left[\begin{array}{c}
	\E \widetilde{\w}_{1,i-1}' \\ 
	\E\widetilde{\boldsymbol{w}}_{2,i-1}' \\ 
	\E\widetilde{\boldsymbol{\lambda}}_{1,i-1}'
	\end{array}\right] - \nonumber\\
	&\!\!\!\!\!\!\!\!\!\!\!\!\!\!\!\!\!\!\!\!\!\!\!\!\!\!\mu \left[\begin{array}{ccc} \mathcal{V}_2^\T \mathcal{H} \mathcal{V}_2 + \eta \mathcal{S}_2^\T \mathcal{S}_2 & \frac{1}{\sqrt{N}} \mathcal{V}_2^\T \mathcal{H} (\mathds{1} \otimes I_M) & \mathcal{S}_2^\T\\
													   \frac{1}{\sqrt{N}} (\mathds{1} \otimes I_M)^\T \mathcal{H} \mathcal{V}_2 & \displaystyle \frac{1}{N} \sum_{k=1}^N R_{u,k} & 0\\
													   -\mathcal{S}_2 & 0 & 0  \end{array}\right] \times \nonumber\\
													   &\quad\ \left[\begin{array}{c}
	\E\widetilde{\w}_{1,i-1}' \\ 
	\E\widetilde{\boldsymbol{w}}_{2,i-1}' \\ 	
	\E \widetilde{\boldsymbol{\lambda}}_{1,i-1}'
	\end{array}\right] \label{eq:matrix_form_random_error_recursion_centralized_original}
\end{align}
Now, observe that when $\mathcal{H} = I_N \otimes R_u$, we have that the above recursion may be simplified due to
\begin{align}
	\frac{1}{\sqrt{N}} \mathcal{V}_2^\T \mathcal{H} (\mathds{1} \otimes I_M) &= \frac{1}{\sqrt{N}} (V_2^\T \otimes I_M) (I_N \otimes R_u) (\mathds{1} \otimes I_M) \nonumber\\
	&= \frac{1}{\sqrt{N}}  (V_2^\T \mathds{1} \otimes R_u)\nonumber\\
	&= 0
\end{align}
since the vectors $V_2$ are orthogonal to the vector $\mathds{1}_N$ since they are all singular vectors of the incidence matrix $C$. Then, recursion \eqref{eq:matrix_form_random_error_recursion_centralized_original} simplifies to
\begin{align}
	\left[\begin{array}{c}
	\E \widetilde{\w}_{1,i}' \\ 
	\E \widetilde{\boldsymbol{w}}_{2,i}' \\ 
	\E \widetilde{\boldsymbol{\lambda}}_{1,i}'
	\end{array}\right]  &=  \left[\begin{array}{c}
	\E \widetilde{\w}_{1,i-1}' \\ 
	\E\widetilde{\boldsymbol{w}}_{2,i-1}' \\ 
	\E\widetilde{\boldsymbol{\lambda}}_{1,i-1}'
	\end{array}\right] - \nonumber\\
	&\!\!\!\!\!\!\!\!\!\!\!\!\!\!\!\!\!\!\!\!\!\!\!\!\!\!\mu \left[\begin{array}{ccc} \mathcal{V}_2^\T \mathcal{H} \mathcal{V}_2 + \eta \mathcal{S}_2^\T \mathcal{S}_2 & 0 & \mathcal{S}_2^\T\\
													   0 & \displaystyle \frac{1}{N} \sum_{k=1}^N R_{u,k} & 0\\
													   -\mathcal{S}_2 & 0 & 0  \end{array}\right] \!\!\left[\begin{array}{c}
	\E\widetilde{\w}_{1,i-1}' \\ 
	\E\widetilde{\boldsymbol{w}}_{2,i-1}' \\ 	
	\E \widetilde{\boldsymbol{\lambda}}_{1,i-1}'
	\end{array}\right] \label{eq:matrix_form_random_error_recursion_centralized}
\end{align}
By examining \eqref{eq:matrix_form_random_error_recursion_centralized}, we arrive at the interpretation that the transformed vector $\E \widetilde{\boldsymbol{w}}_{2,i}'$ is the mean-error of the primal \emph{reference} recursion:
\begin{align}
	\boldsymbol{w}_{2,i}' = \boldsymbol{w}_{2,i-1}' - \frac{\mu}{N} \sum_{k=1}^N \u_{k,i}^\T (\d_{k}(i) - \u_{k,i} \boldsymbol{w}_{2,i-1}')
\end{align}
This interpretation is consistent with the analysis performed in \cite{jianshu_part1,jianshu_part2} for diffusion and consensus networks. On the other hand, the vector $\E \widetilde{\w}_{1,i}'$ measures node-specific discrepancies, and the dual variable $\E \widetilde{\boldsymbol{\lambda}}_{1,i}'$ only depends on these discrepancies since it seeks to correct them by forcing all nodes to arrive at the same estimate. \qed

\subsection{Mean-Square-Error Recursion}
Different choices for $\Sigma$ in \eqref{eq:MSD_equivalence} allow us to evaluate different performance metrics. For example, when the network mean-square-deviation (MSD) is desired, we set $\Sigma = I_{MN}$. Using \eqref{eq:MSD_equivalence}, we shall instead examine $\E \|\widetilde{\w}_{i}'\|^2_{\Sigma'}$. We first lift $\Sigma'$ to
\begin{align}
	\Gamma' \triangleq \left[\begin{array}{c|c}
	\Sigma' & 0_{NM \times (N-1)M} \\ \hline
	0_{(N-1)M \times NM} & 0_{(N-1)M\times (N-1)M}
	\end{array} \right]
\end{align}
so that the extended model \eqref{eq:matrix_form_random_error_recursion} can be used to evaluate $\E \|\widetilde{\w}_{i}'\|^2_{\Sigma'}$ as follows:
\begin{align}
&\E \left\|\left[\begin{array}{c}
\widetilde{\w}_{1,i}' \\
\widetilde{\w}_{2,i}'
\end{array}\right]\right\|^2_{\Sigma'} \!= \E \left\|\left[\begin{array}{c}
	\widetilde{\w}_{1,i}' \\
	\widetilde{\w}_{2,i}' \\
	\widetilde{\blambda}_{1,i}'
	\end{array}\right]\right\|^2_{\Gamma'} \nonumber\\
&\!=\! \E \left(\!\left\|\!\!\left[\begin{array}{c}
	\widetilde{\w}_{1,i-1}' \\
	\widetilde{\w}_{2,i-1}' \\ 
	\widetilde{\blambda}_{1,i-1}'
	\end{array}\right]\!\!\right\|^2_{{{\boldsymbol{\mathcal{B}}_i'^\T}}\Gamma' \boldsymbol{\mathcal{B}}'_i} \!\right) \!\!+\! \mu^2 \E \left\|\!\!\left[\begin{array}{c}
	\mathcal{V}_2^\T \z_i \\ 
	\mathcal{V}_0^\T \z_i \\ 
	\mathbbold{0}_{(N-1)M}
	\end{array} \right]\!\!\right\|^2_{\Gamma'} \label{eq:MSE_recursion_1}
\end{align}
Following arguments similar to \cite[pp.~381--383]{sayed2012diffbookchapter}, it can be verified that 
\begin{align}
	\E\{{{{\boldsymbol{\mathcal{B}}_i'^\T}}\Gamma' \boldsymbol{\mathcal{B}}'_i}\} &= \mathcal{B}'^\T \Gamma' \mathcal{B}' + O(\mu^2)
\end{align}
We will be assuming sufficiently small step-sizes (which correspond to a slow adaptation regime) so that we can ignore terms that depend on higher-order powers of $\mu$. Arguments in \cite{SayedProcIEEE} show that conclusions obtained under this approximation lead to performance results that are accurate to first-order in the step-size parameters and match well with actual performance for small step-sizes. Therefore, we approximate $\E\{{{{\boldsymbol{\mathcal{B}}'^\T}_i}\Gamma' \boldsymbol{\mathcal{B}}'_i}\} \approx \mathcal{B}'^\T \Gamma' \mathcal{B}'$ and replace \eqref{eq:MSE_recursion_1} by 
\begin{equation}
	\E \left\|\!\!\left[\begin{array}{c}
	\widetilde{\w}_{1,i}' \\
	\widetilde{\w}_{2,i}' \\ 
	\widetilde{\blambda}_{1,i}'
	\end{array}\right]\!\!\right\|^2_{\Gamma'} \!\!\!\!\!\approx\! \E \left\|\!\!\left[\begin{array}{c}
	\widetilde{\w}_{1,i-1}' \\
	\widetilde{\w}_{2,i-1}' \\ 
	\widetilde{\blambda}_{1,i-1}'
	\end{array}\right]\!\!\right\|^2_{{\mathcal{B}}'^\T \Gamma' \mathcal{B}'} \!\!\!\!\!\!\!\!\!\!\!\!+ \mu^2 \E \left\|\!\!\left[\begin{array}{c}
	\mathcal{V}_2^\T \z_i \\ 
	\mathcal{V}_0^\T \z_i \\ 
	\mathbbold{0}_{(N-1)M}
	\end{array} \right]\!\!\right\|^2_{\Gamma'} \label{eq:MSE_recursion_approx1}
\end{equation}
which, by using properties of the Kronecker product operation, can be rewritten in the equivalent form:
\begin{align}
	\E \left\|\left[\begin{array}{c}
	\widetilde{\w}_{1,i}' \\
	\widetilde{\w}_{2,i}' \\ 
	\widetilde{\blambda}_{1,i}'
	\end{array}\right]\right\|^2_{\gamma'} &\approx \E \left\|\left[\begin{array}{c}
	\widetilde{\w}_{1,i-1}' \\
	\widetilde{\w}_{2,i-1}' \\ 
	\widetilde{\blambda}_{1,i-1}'
	\end{array}\right]\right\|^2_{\mathcal{F}' \gamma'} \!\!\!\!\!\!+\! \mu^2 \left(\mathrm{bvec}(R_h^\T)\right)^\T \!\!\gamma' \label{eq:vec_weighted_expression}
\end{align}
In this second form, we are using the notation $\|x\|^2_{\Sigma}$ and $\|x\|^2_{\sigma}$ interchangeably in terms of the weighting matrix $\Sigma$ or its vectorized form, $\sigma=\mbox{\rm bvec}(\Sigma)$, where $\mathrm{bvec}(X)$ denotes block vectorization \cite{cassio}. In block vectorization, each $M\times M$ submatrix of $X$ is vectorized, and the vectors are stacked on top of each other. Moreover, we are introducing
\begin{align}
	\gamma' &\triangleq \mathrm{bvec}(\Gamma')\\
	\mathcal{F}' &\triangleq \mathcal{B}'^\T \otimes_b \mathcal{B}'^\T \label{eq:script_F'}\\
	R_h &\triangleq \E \left[\begin{array}{c}
	\mathcal{V}_2^\T \z_i \\ 
	\mathcal{V}_0^\T \z_i \\ 
	\mathbbold{0}_{(N-1)M}
	\end{array} \right] \left[\begin{array}{c}
	\mathcal{V}_2^\T \z_i \\ 
	\mathcal{V}_0^\T \z_i \\ 
	\mathbbold{0}_{(N-1)M}
	\end{array} \right]^\T\\
	&= \left[\!\!\begin{array}{cc|c} 
	            \mathcal{V}_2^\T \mathcal{R}_{z} \mathcal{V}_2 & 0_{(N-1) M \!\times\! M} & 0_{(N-1) M} \\
	            0_{M \!\times\! (N-1) M} & \displaystyle \frac{1}{N} \!\sum_{k=1}^N \!\sigma_{v,k}^2 R_{u,k} & 0_{M \!\times\! (N-1) M}\\\hline
	            0_{(N-1) M} & 0_{(N-1) M \!\times\! M} & 0_{(N-1)M}
	         \end{array}\!\!\right]
\end{align}
where $\otimes_b$ denotes the block Kronecker product \cite{TracySingh} and 
	\begin{equation}
		\mathcal{R}_z \triangleq \E[\z_i \z_i^\T] = \mathrm{blockdiag}\{\sigma_{v,k}^2 R_{u,k}\} \label{eq:R_z}
	\end{equation}

\section{Stability Analysis}
\label{sec:Stability}
Using the just derived mean and mean-square update relations, we start by analyzing the stability of the algorithm \eqref{eq:AL_primal_alg}--\eqref{eq:AL_dual_alg}.   We first review the following concepts.
\begin{definition}[Hurwitz Matrix]
	A real square matrix is called Hurwitz if all its eigenvalues possess negative real parts \cite{antsaklis2005linear,kailath1980linear}. \qed
\end{definition}

\begin{definition}[Stable Matrix]
	A real square matrix is called (Schur) stable if all its eigenvalues lie inside the unit circle \cite{antsaklis2005linear,kailath1980linear}. \qed
\end{definition}
\noindent The following lemma relates the two concepts \cite[p.~39]{poliak1987introduction}.
\begin{lemma}[Hurwitz and stable matrices]
	\label{lem:discrete-time-stability}
	Let $A$ be a Hurwitz matrix. Then, the matrix $B = I + \mu A$ is stable if, and only if,
	\begin{align}
		0 < \mu < \min_j \left\{-2 \frac{\mathfrak{Re}\{\lambda_j(A)\}}{|\lambda_j(A)|^2}\right\} \label{eq:general_step-size_range}
	\end{align}
	where $\lambda_j(A)$ is the $j$-th eigenvalue of the matrix $A$. \qed
\end{lemma}

\subsection{Mean Stability}
In order to show that $\E \widetilde{\w}_i'\rightarrow \mathbbold{0}_{NM}$, we need to show that the matrix $\mathcal{B}'$ in \eqref{eq:B_prime} is stable. To establish this fact, we rely on Lemma \ref{lem:discrete-time-stability} and on the following two auxiliary results. 
\begin{lemma}[Hurwitz stability of a matrix]
	\label{lem:hurwitz}
	Let a block square matrix $G$ have the following form:
	\begin{align}
		G = -\left[\begin{array}{cc}
		X & Y^\T \\ 
		-Y & 0_{Q \times Q}
		\end{array} \right] \label{eq:lem:Hurwitz_stability_matrix_W}
	\end{align}
	and let the matrix $X \in \mathbb{R}^{P \times P}$ be positive-definite and $Y \in \mathbb{R}^{Q \times P}$ possess full row rank. Then, the matrix $G$ is Hurwitz.
\end{lemma}
\begin{proof}
	The argument follows a similar procedure to the proof of Proposition 4.4.2 in \cite[p.~449]{bertsekas1999nonlinear}. Let $\lambda_j$ denote the $j$-th eigenvalue of $G$ and let the corresponding eigenvector be $g_j = [g_{1,j},g_{2,j}]^\T \neq \mathbbold{0}_{Q+P}$ so that $G g_j = \lambda_j g_j$, where $g_{1,j} \in \mathbb{C}^{P\times 1}$ and $g_{2,j} \in \mathbb{C}^{Q\times 1}$. Then, we have that
	\begin{align}
		\mathfrak{Re}\left\{g_j^* G g_j \right\} &= \mathfrak{Re}\left\{\lambda_j g_j^* \cdot g_j \right\} \nonumber\\
												 &= \mathfrak{Re}\left\{\lambda_j \right\} (\|g_{1,j}\|^2 + \|g_{2,j}\|^2) \label{eq:lem:Hurwitz_stability_first_form}
	\end{align}
	Similarly, using \eqref{eq:lem:Hurwitz_stability_matrix_W}, we have that the same quantity is given by
	\begin{align}
		\mathfrak{Re}\left\{g_j^* G g_j \right\} &= \mathfrak{Re}\left\{ -g_{1,j}^* X g_{1,j} \!-\! g_{1,j}^* Y^\T g_{2,j} \!+\! g_{2,j}^* Y g_{1,j} \right\} \nonumber\\
		&= -\mathfrak{Re}\left\{ g_{1,j}^* X g_{1,j} \right\} \label{eq:lem:Hurwitz_stability_second_form}
	\end{align}
	since
	\begin{equation}
		\mathfrak{Re}\{g_{2,j}^* Y g_{1,j}\} \!=\! \mathfrak{Re}\{(g_{2,j}^* Y g_{1,j})^*\} \!=\! \mathfrak{Re}\{g_{1,j}^* Y^\T g_{2,j}\}
	\end{equation}
	Now, combining \eqref{eq:lem:Hurwitz_stability_first_form}--\eqref{eq:lem:Hurwitz_stability_second_form}, we have that
	\begin{align}
		\mathfrak{Re}\left\{\lambda_j \right\} (\|g_{1,j}\|^2 + \|g_{2,j}\|^2) &=  -\mathfrak{Re}\left\{ g_{1,j}^* X g_{1,j} \right\} \label{eq:proof_hurwitz_second_last_eq}
	\end{align}
	Since the matrix $X$ is positive-definite, then either 1) $\mathfrak{Re}\left\{\lambda_j \right\} < 0$ and $g_{1,j} \neq \mathbbold{0}_P$ or 2) $\mathfrak{Re}\left\{\lambda_j \right\} = 0$ and $g_{1,j} = \mathbbold{0}_P$. Suppose now that $g_{1,j} = \mathbbold{0}_P$, then 
	\begin{align}
		G g_j = \lambda_j g_j \Rightarrow -Y^\T g_{2,j} = \lambda_j g_{1,j} = \mathbbold{0}_P
	\end{align}
	but since we assumed that $g_j \neq \mathbbold{0}_{Q+P}$ while $g_{1,j} = \mathbbold{0}_P$, then we have that $g_{2,j} \neq \mathbbold{0}_Q$. This implies that $g_{2,j} \neq \mathbbold{0}_Q$ is in the nullspace of $Y^\T$, which is not possible since $Y$ has full row rank. Therefore $g_{1,j} \neq \mathbbold{0}_{P}$ and we conclude that $\mathfrak{Re}\left\{\lambda_j \right\} < 0$ and thus the matrix $G$ is Hurwitz.
\end{proof}
We can also establish the following result regarding the positive-definiteness of $\mathcal{V}^\T \mathcal{H} \mathcal{V}+\eta \sLambda$.
\begin{lemma}[Positive-definiteness of $\mathcal{V}^\T \mathcal{H} \mathcal{V}+\eta \sLambda$]
	\label{lem:pos_def_v'Hv}
	Let the sum of regressor covariance matrices satisfy \eqref{eq:partial_observability}, and let the network be connected. Then, there exists $\bar{\eta} \geq 0$ such that for all $\eta \!>\! \bar{\eta}$, we have that $\mathcal{V}^\T \mathcal{H} \mathcal{V}\!+\!\eta \sLambda$ is positive-definite.
\end{lemma}
\begin{proof}
	First, note that 
\begin{align}
	\mathcal{V}^\T \mathcal{H} \mathcal{V} \!+\! \eta \sLambda &= \left[\!\!\begin{array}{cc}
	\mathcal{V}_2^\T \mathcal{H} \mathcal{V}_2 + \eta \mathcal{S}_2^\T \mathcal{S}_2& \mathcal{V}_2^\T \mathcal{H} \mathcal{V}_0 \\ 
	\mathcal{V}_0^\T \mathcal{H} \mathcal{V}_2 & \displaystyle \frac{1}{N} \sum_{k=1}^N R_{u,k}
	\end{array} \!\!\right] \label{eq:y'Hy}
\end{align}
To show that \eqref{eq:y'Hy} is positive-definite, it is sufficient to show that the Schur complement relative to the lower-right block is positive definite. This Schur complement is given by $\eta \mathcal{S}_2^\T \mathcal{S}_2 + \mathcal{Z}$ where we defined
\begin{align}
	\mathcal{Z} \triangleq \mathcal{V}_2^\T \mathcal{H} \mathcal{V}_2 - \mathcal{V}_2^\T \mathcal{H} \mathcal{V}_0 \left(\frac{1}{N} \sum_{k=1}^N R_{u,k}\right)^{-1} \mathcal{V}_0^\T \mathcal{H} \mathcal{V}_2 \label{eq:D}
\end{align}
Then, by Weyl's inequality \cite[p.~181]{HornJohnsonVol1}, we have that the Schur complement is positive-definite when
\begin{align}
	\lambda_{\min}(\eta \mathcal{S}_2^\T \mathcal{S}_2 + \mathcal{Z}) &\geq \eta \lambda_{\min}(\mathcal{S}_2^\T \mathcal{S}_2) + \lambda_{\min}(\mathcal{Z}) > 0 
\end{align}
which is guaranteed when 
\begin{align}
	\eta >  \frac{-\lambda_{\min}(\mathcal{Z})}{\lambda_{\min}(\mathcal{S}_2^\T \mathcal{S}_2)} \label{eq:eta_bar}
\end{align}
where $\lambda_{\min}(\mathcal{S}_2^\T \mathcal{S}_2) > 0$ is the second-smallest eigenvalue of the Laplacian matrix (algebraic connectivity of the topology graph or Fiedler value \cite{fiedler1973algebraic,fiedler1975property,Bertrand20131106,BertrandSPM}) and is positive when the network is connected.
\end{proof}

Using the preceding lemmas, we are now ready to prove the mean-stability of algorithm \eqref{eq:AL_primal_alg}--\eqref{eq:AL_dual_alg} for large $\eta > 0$.
\begin{theorem}[Mean stability of the AL algorithm]
\label{thm:generalTopology_mean_stability_partial_observability}
	Under \eqref{eq:partial_observability} and over connected networks, there exists $\bar{\eta}$ such that for all $\eta > \bar{\eta}$, the matrix $\mathcal{B}'$ is stable, i.e., $\rho(\mathcal{B}') \!<\! 1$ for small $\mu$.
\end{theorem}
\begin{proof}
	From Lemma \ref{lem:pos_def_v'Hv}, we have that $\mathcal{V}^\T \mathcal{H} \mathcal{V} + \eta \sLambda$ is positive-definite for large $\eta$. Using \eqref{eq:R_prime} and \eqref{eq:y'Hy} we write 
	\begin{align}
		-\mathcal{R}' &= 	-\left[\!\!\begin{array}{cc|c} \multicolumn{2}{c|}{\multirow{2}{*}{$\mathcal{V}^\T \mathcal{H} \mathcal{V} + \eta \sLambda$}} & \mathcal{S}_2^\T\\
													  \multicolumn{2}{c|}{}  & 0_{M\!\times\!(N-1)M}\\\hline
													   -\mathcal{S}_2 & 0_{(N-1)M\!\times\!M} & 0_{(N-1)M\!\times\!(N-1)M}  \end{array}\!\!\right]									   \label{eq:R'_mean_stability}
	\end{align}
	The matrix $-{\cal R}'$ so defined is in the same form required by Lemma \ref{lem:hurwitz} where the bottom-left block has full-row-rank since $\mathcal{S}_2$ is invertible. We conclude from Lemma \ref{lem:hurwitz} that $-\mathcal{R}'$ is Hurwitz. Then, by Lemma \ref{lem:discrete-time-stability}, there exists some range for $0 < \mu < \bar{\mu}$ for which the matrix ${\cal B}'=I-\mu {\cal R}'$ is stable, where
	\begin{align}
		\bar{\mu} = \min_j \left\{2 \frac{\mathfrak{Re}\{\lambda_j(\mathcal{R}')\}}{|\lambda_j(\mathcal{R}')|^2}\right\} \label{eq:mu_bar}
	\end{align}
\end{proof}

We conclude that the AL algorithm \eqref{eq:AL_primal_alg}--\eqref{eq:AL_dual_alg} with $\eta > 0$ is mean stable under partial observation conditions (i.e., when some of the $R_{u,k}$ may be singular but the aggregate sum \eqref{eq:partial_observability} is still positive-definite). Observe though that the result of Theorem \ref{thm:generalTopology_mean_stability_partial_observability} does not necessarily apply to the AH algorithm since for that algorithm, $\eta = 0$, while the bound on the right-hand side of \eqref{eq:eta_bar} can be positive (in fact, there are cases in which the matrix $-\mathcal{R}'$ will not be Hurwitz when $\eta = 0$---see App.~\ref{app:AH_partial_observability}). It is nevertheless still possible to prove the stability of the AH algorithm under the more restrictive assumption of a positive-definite ${\cal H}$ (which requires all individual $R_{u,k}$ to be positive-definite rather than only their sum as we have assumed so far in \eqref{eq:partial_observability}).

\begin{corollary}[Mean stability of the AL and AH algorithms]
\label{cor:generalTopology_mean_stability}
	Let the matrix $\mathcal{H}$ be positive-definite. Furthermore, let the network be connected. Then, the matrix $\mathcal{B}'$ is stable, i.e., $\rho(\mathcal{B}') < 1$ for small enough step-sizes.
\end{corollary}
\begin{proof}
	Since $\mathcal{H}$ is now assumed positive-definite, we have that $\mathcal{V}^\T \mathcal{H} \mathcal{V} + \eta \sLambda$ in \eqref{eq:R'_mean_stability} is positive-definite for any $\eta \geq 0$. We may then appeal to Lemma \ref{lem:hurwitz} to conclude that the matrix $-\mathcal{R}'$ is Hurwitz, and by Lemma \ref{lem:discrete-time-stability}, there exists some range for $0 < \mu < \bar{\mu}$ so that  $\mathcal{B}'$ is stable, where $\bar{\mu}$ is given by \eqref{eq:mu_bar}.
\end{proof}

The assumption that the matrix $\mathcal{H}$ is positive-definite is only satisfied when all regressor covariance matrices $R_{u,k}$ are positive-definite. We observe, therefore, that the AH algorithm cannot generally solve the partial observation problem in which only the sum of the covariance matrices is positive-definite and not each one individually. Furthermore, the AL algorithm may not be able to solve this problem either unless the regularizer $\eta$ is large enough.

\subsection{Mean-Square Stability}

\begin{theorem}[Mean-square stability]
	\label{thm:generalTopology_mean_square_stability_partial_observability}
	Under the same conditions of Theorem \ref{thm:generalTopology_mean_stability_partial_observability}, there exists some $\bar{\eta}$ such that for all $\eta > \bar{\eta}$, the matrix $\mathcal{F}'$ is stable, i.e., $\rho(\mathcal{F}') < 1$ for small step-sizes.
\end{theorem}
\begin{proof}
	We know that $\rho(\mathcal{F}') \!=\! \rho(\mathcal{B}'^\T \!\otimes_b\! \mathcal{B}'^\T) \!=\! \rho(\mathcal{B}')^2$.	By Thm.~\ref{thm:generalTopology_mean_stability_partial_observability}, $\rho(\mathcal{B}') \!<\! 1$, and we have that $\rho(\!\mathcal{F}'\!) \!<\! 1$ for small $\mu$.
\end{proof}

Therefore, when the step-size is sufficiently small, the AL algorithm can be guaranteed to converge in the mean and mean-square senses under partial observation for large enough $\eta$. We can similarly establish an analogous result to Corollary \ref{cor:generalTopology_mean_stability} for the AH algorithm under the more restrictive assumption of positive-definite $\mathcal{H}$.

\begin{corollary}[Mean-square stability of the AL and AH algorithms]
	\label{cor:generalTopology_mean_square_stability}
	Under the same conditions of Corollary \ref{cor:generalTopology_mean_stability}, then the matrix $\mathcal{F}'$ is stable, i.e., $\rho(\mathcal{F}') < 1$ for small step-sizes.
\end{corollary}
\begin{proof}
	The argument is similar to Theorem \ref{thm:generalTopology_mean_square_stability_partial_observability} by noting that ${\cal B}'$ is also stable. 
\end{proof}

Next, we will examine the step-size range \eqref{eq:mu_bar} required for convergence, and we will see that the AL and AH algorithms are not as stable as non-cooperative and diffusion strategies.

\subsection{Examination of Stability Range}
\label{ssec:stepsize_range}
In order to gain insight into the step-size range defined by \eqref{eq:mu_bar}, we will analyze the eigenvalues of the matrix $\mathcal{R}'$ in the case when $\mathcal{H} = I_N \otimes R_u$, where $R_u > 0$.

\begin{theorem}[Eigenvalues of $\mathcal{R}'$]
\label{thm:eigenvalues_minusR2}
Assuming $\mathcal{H} = I_N \otimes R_u$ where $R_u$ is positive-definite, the $(2N -1)M$ eigenvalues of the matrix $\mathcal{R}'$ are given by $\{\lambda_{\ell,k}(\mathcal{R}')\} = \sigma \cup \tau$, where 
\begin{align}
	\sigma &\triangleq \left\{\lambda_\ell(R_u): 1\leq \ell \leq M\right\} \label{eq:eigenvalues_sigma}\\
	\tau &\triangleq  \left\{\frac{1}{2} (\lambda_\ell(R_u) + \eta \lambda_k(L)) \pm \right. \label{eq:eigenvalues_tau}\\
		&\!\!\!\!\!\left.\frac{1}{2} \sqrt{(\lambda_\ell(R_u) \!+\! \eta \lambda_k(L))^2 \!-\! 4 \lambda_k(L)}: 1 \!\leq\! k \!\leq\! \!N\!-\!1, 1\!\leq\! \ell\! \leq\! M\right\} \nonumber
\end{align}
where $L$ is the Laplacian matrix of the network topology.
\end{theorem}
\begin{proof}
%	The expressions follow directly from determining the roots of the characteristic polynomial $\det (xI-{\cal R}')$ using the block determinantal formula \cite[p.~5]{laub}.
	To obtain the eigenvalues of $\mathcal{R}'$, we will solve for $\lambda_{k,\ell}$ using  
	\begin{align}
		\mathrm{det}(\mathcal{R}' - \lambda_{k,\ell} I) = 0
	\end{align}	
	To achieve this, we call upon Schur's determinantal formula for a block matrix \cite[p.~5]{laub}:
	\begin{align}
		\mathrm{det}\left[\begin{array}{cc}
		A & B \\ 
		C & D
		\end{array} \right] = \mathrm{det}(D) \mathrm{det}(A - B D^{-1} C) \label{eq:determinant_of_block_matrix}
	\end{align}
to note, using \eqref{eq:R'_mean_stability}, that $\mathrm{det}(\mathcal{R}'-\lambda_{k,\ell} I)$ is given by
\begin{align}
	&\det\!\left[\!\begin{array}{cc|c} \multicolumn{2}{c|}{\multirow{2}{*}{$\mathcal{V}^\T \mathcal{H} \mathcal{V} + \eta \sLambda- \lambda_{k,\ell} I_{MN}$}} & \mathcal{S}_2^\T\\
													  \multicolumn{2}{c|}{}  & 0_{M\!\times\!(N-1)M}\\\hline
													   -\mathcal{S}_2 & 0_{(N-1)M\!\times\!M} & - \lambda_{k,\ell} I_{(N-1)M}  \end{array}\!\right] =\nonumber\\
		&\det(-\lambda_{k,\ell} I_{(N-1)M}) \det\left(\!(I_N \otimes R_u) \!+\! \eta \sLambda \!-\! \lambda_{k,\ell} I_{MN} \!-\! \frac{\sLambda}{\lambda_{k,\ell}} \!\right) \label{eq:evals_intermediate_step1}
\end{align}
	where we used \eqref{eq:sLambda} and the fact that 
	\begin{equation}
		\mathcal{V}^\T \!(I_N \!\otimes\! R_u) \mathcal{V} \!=\! (V^\T \!\!\otimes\! I_M)\! (I_N \!\otimes\! R_u)\! (V \!\otimes\! I_M) \!=\! I_N \!\otimes\! R_u
	\end{equation}
	We already demonstrated in Corollary \ref{cor:generalTopology_mean_stability} that the matrix $-\mathcal{R}'$ is Hurwitz and, therefore, $\lambda_{k,\ell} \neq 0$. We conclude that $\det(-\lambda_{k,\ell} I_{(N-1)M})\neq 0$ and we focus on the last term in \eqref{eq:evals_intermediate_step1}. Observe that the matrix $(I_N \otimes R_u) + \eta \sLambda - \lambda_{k,\ell} I_{MN} - \frac{1}{\lambda_{k,\ell}} \sLambda$ is block-diagonal, with the $k$-th block given by $R_u + \left(-\lambda_{k,\ell} + \left(\eta - \frac{1}{\lambda_{k,\ell}}\right) \lambda_k(L)\right) I_M$ since the matrix $\sLambda$ is a diagonal matrix with diagonal blocks $\lambda_{k}(L) \otimes I_M$, where $L$ is the Laplacian matrix of the network (see \eqref{eq:sLambda}). Therefore, since the smallest eigenvalue, $\lambda_N(L)$, of $L$ is zero, we obtain
	\begin{align}
		&\det\left((I_N \otimes R_u) - \lambda_{k,\ell} I_{MN} + \left(\eta - \frac{1}{\lambda_{k,\ell}}\right) \sLambda\right) \nonumber\\
%		&= \det\left\lbrace\left[\begin{array}{cccc}
%		-R_u -\lambda_{k,\ell} I_M&  &  &  \\ 
%		 & -R_u -\lambda_{k,\ell} I_M - (\eta + \frac{1}{\lambda_{k,\ell}}) \lambda_1(L) I_M &  &  \\ 
%		 &  & \ddots &  \\ 
%		 &  &  & -R_u -\lambda_{k,\ell} I_M - (\eta + \frac{1}{\lambda_{k,\ell}}) \lambda_N(L) I_M
%		\end{array} \right]\left\rbrace \nonumber\\
		&= \det ( R_u -\lambda_{N,\ell} I_M ) \times \\
		&\quad\ \prod_{k=1}^{N-1} \det \left( R_u + \left(-\lambda_{k,\ell} + \left(\eta - \frac{1}{\lambda_{k,\ell}}\right) \lambda_k(L)\right) I_M \right) \nonumber
	\end{align}
	It follows, by setting the above to zero, that $M$ of the eigenvalues of ${\cal R}'$, $\{\lambda_{1,1},\ldots,\lambda_{1,M}\}$, are given by:
	\begin{align}
		\lambda_{N,\ell} = \lambda_\ell(R_u)
	\end{align}
	The remaining $2 (N - 1) M$ eigenvalues are obtained by solving algebraic equations of the form:
	\begin{align}
		\det\left( \Delta +\left(-\lambda_{k,\ell} + \left(\eta - \frac{1}{\lambda_{k,\ell}}\right) \lambda_k(L)\right) I_M \right) = 0
	\end{align}
	where $\Delta$ is a diagonal matrix with the eigenvalues of $R_u$ along its diagonal.
%	, since 
%	\begin{align}
%		&\det\left\lbrace R_u + \left(-\lambda_{k,\ell} + \left(\eta - \frac{1}{\lambda_{k,\ell}}\right) \lambda_k(L)\right) I_M\right\rbrace \nonumber\\
%		&= \det\left\lbrace \Phi \Delta \Phi^\T + \left(-\lambda_{k,\ell} + \left(\eta - \frac{1}{\lambda_{k,\ell}}\right) \lambda_k(L)\right) I_M\right\rbrace \nonumber\\
%		&\stackrel{(a)}{=} \det\left\lbrace \Delta + \left(-\lambda_{k,\ell} + \left(\eta - \frac{1}{\lambda_{k,\ell}}\right) \lambda_k(L)\right) \Phi^\T \Phi\right\rbrace \nonumber\\
%		&= \det\left\lbrace \Delta + \left(-\lambda_{k,\ell} + \left(\eta - \frac{1}{\lambda_{k,\ell}}\right) \lambda_k(L)\right) I_M\right\rbrace
%	\end{align}		
%	where $R_u = \Phi \Delta \Phi^\T$ is the eigen-decomposition of $R_u$ with orthogonal matrix $\Phi \in \mathbb{R}^{M\times M}$ and step $(a)$ is due to $\det(A) = \det(\Phi^{-1} A \Phi)$. 
	This yields
	\begin{align}
		\lambda_{k,\ell}^2 - (\lambda_\ell(R_u) + \eta \lambda_k(L)) \lambda_{k,\ell} + \lambda_k(L) &= 0
	\end{align}
	Solving the above quadratic equation, we obtain the remaining $2 (N-1) M$ eigenvalues given by expression \eqref{eq:eigenvalues_tau}.
\end{proof}

\begin{figure*}[t]
\centering
\subfigure[$\mu = 1.1$]{\includegraphics[width=0.28\textwidth]{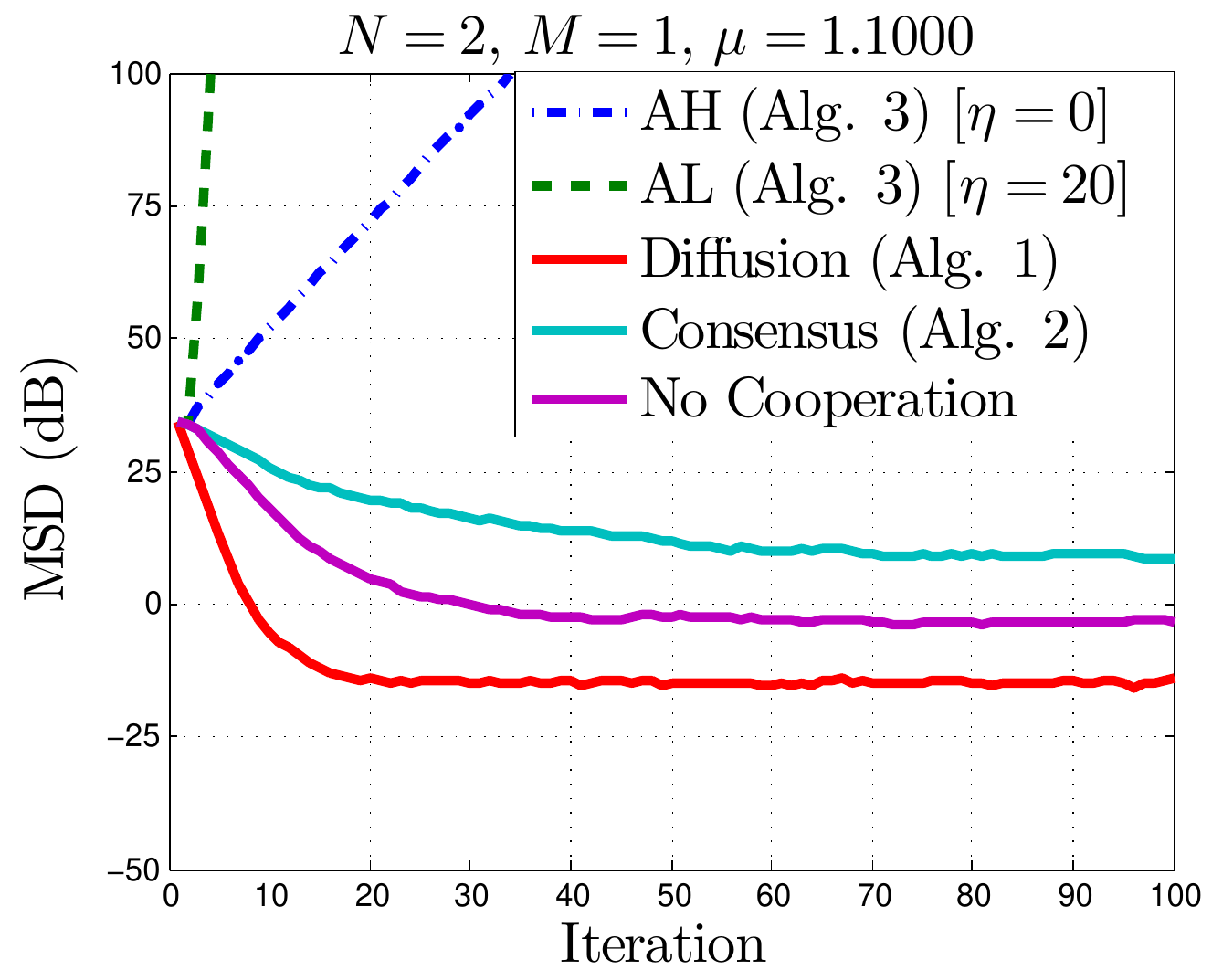}}
\quad\quad
\subfigure[$\mu = 0.75$]{\includegraphics[width=0.28\textwidth]{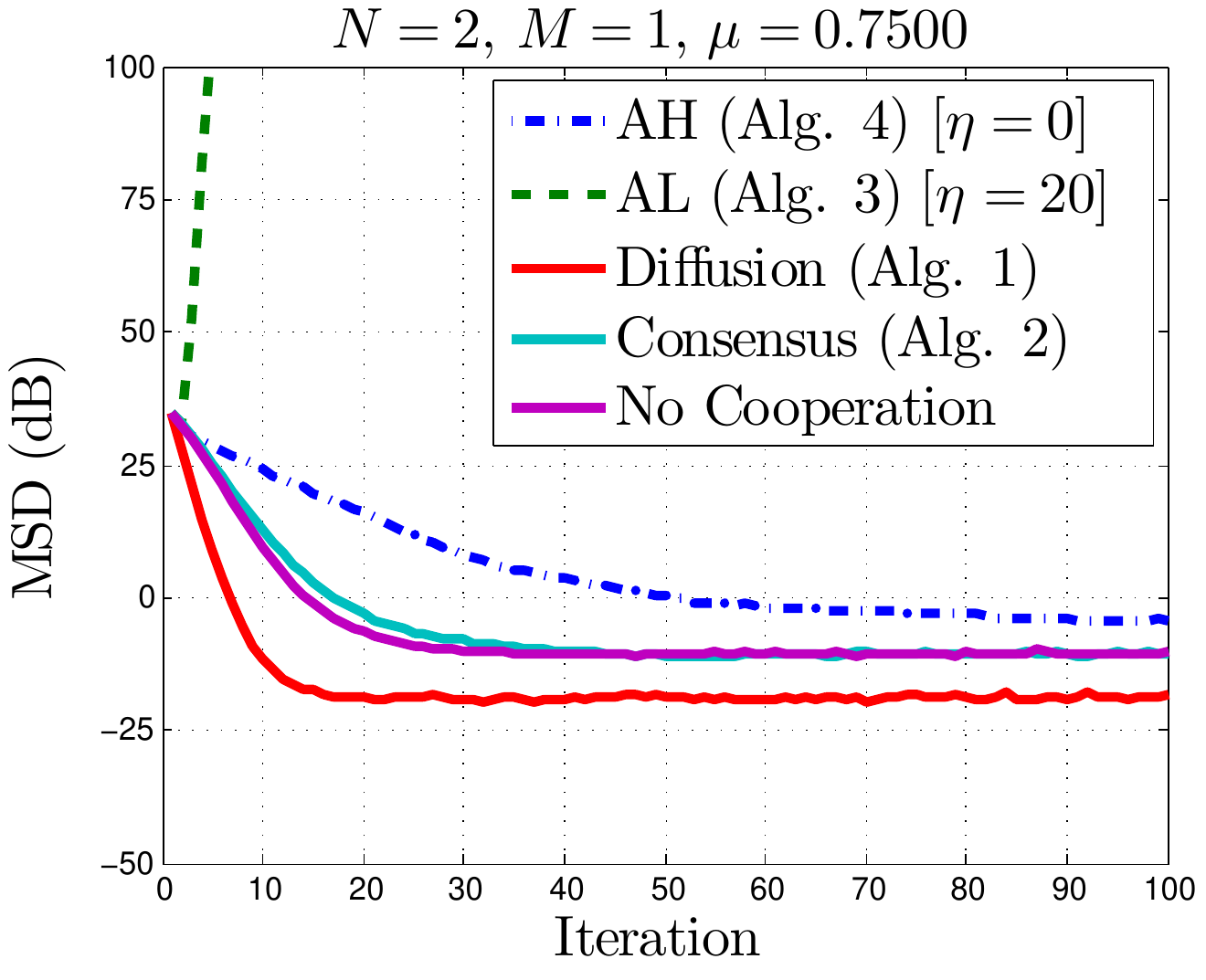}}
\quad\quad
\subfigure[$\mu = 0.0469$]{\includegraphics[width=0.28\textwidth]{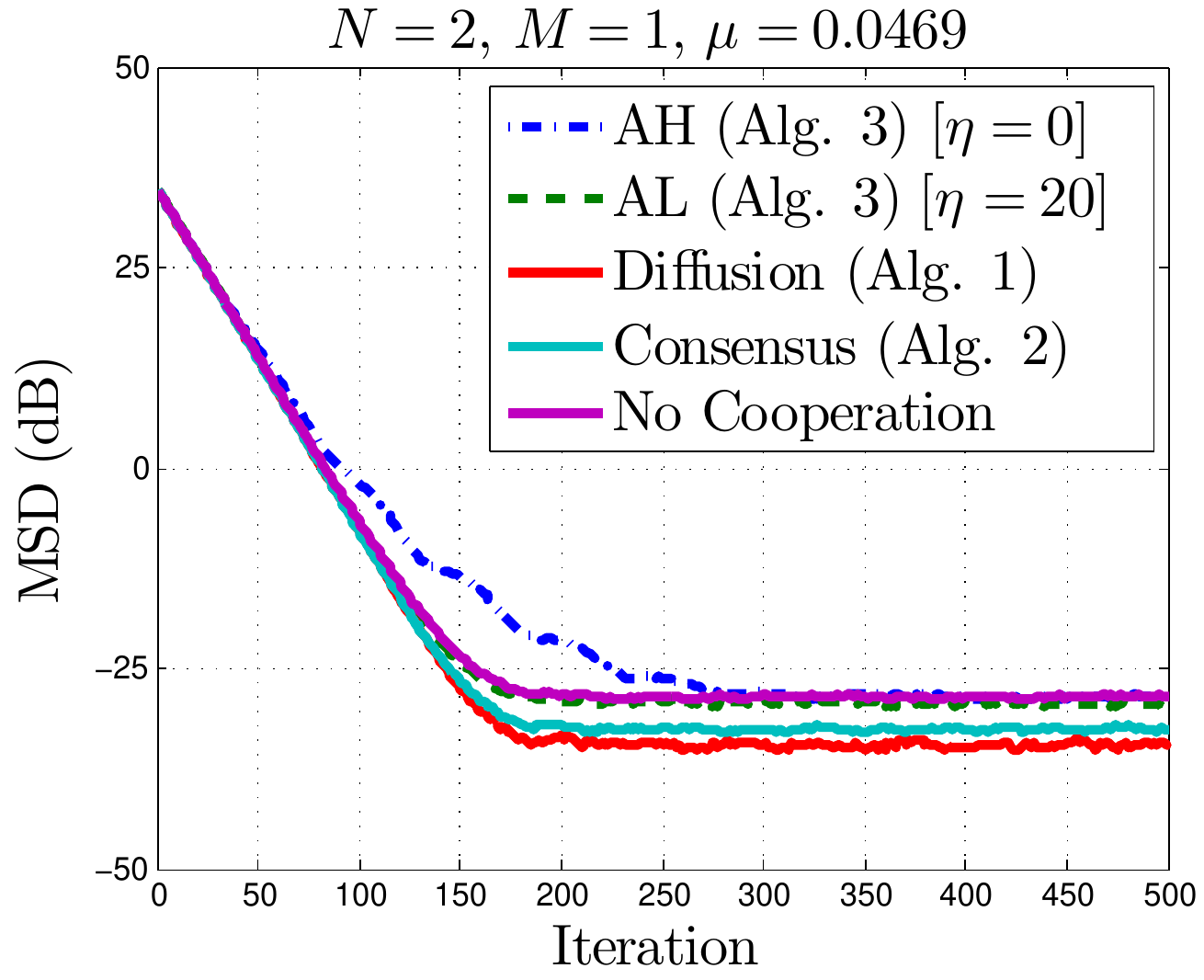}}
\caption{(a) Simulated example where diffusion, consensus, and non-cooperative algorithms converge but the AL ($\eta = 20.0$) and AH algorithms fail to converge when $\mu = 11/10$; (b) Same example with $\mu = 3/4$, and thus the AH algorithm converges, but not the AL algorithm; (c) All algorithms converge when $\mu = 3/64$.}
\label{fig:example_divergence}
\end{figure*}

\noindent Since we have shown that $\lambda_\ell(R_u)$ are always eigenvalues of $\mathcal{R}'$, we have that the step-size range required for convergence of the AL algorithm is always bounded above (and, hence, smaller) than the stability bound for the non-cooperative and diffusion algorithms when $\mathcal{H} = I_N \otimes R_u$ (see \eqref{eq:diff_stability_range}). 

Another aspect we need to consider is how the stability range \eqref{eq:mu_bar} depends on the regularization parameter $\eta$ and the network topology. It is already known that the mean stability range for diffusion strategies \eqref{eq:diff_stability_range} is independent of the network topology \cite{sayed2013diffusion}. In contrast, it is also known that the stability range for consensus strategies \eqref{eq:consensus_1}--\eqref{eq:consensus_2} is dependent on the network topology \cite{tu2012diffusion,SayedProcIEEE}. We are going to see that the stability of the AL algorithm for large $\eta$ is also dependent on the network topology. Indeed, as $\eta\rightarrow\infty$, we have from \eqref{eq:eigenvalues_tau} that some of the eigenvalues of $\mathcal{R}'$, besides the ones fixed at $\lambda_\ell(R_u)$, will approach $\eta \lambda_k(L)$. This means that the step-size range \eqref{eq:mu_bar} required for convergence will need to satisfy:
	\begin{align}
		\mu < \min_{1\leq k\leq N-1} \left\{\frac{2}{\eta \lambda_k(L)}\right\} = \frac{2}{\eta \lambda_1(L)} \label{eq:convergence_condition_large_eta}
	\end{align}
	where $\lambda_{1}(L)$ denotes the largest eigenvalue of $L$. Clearly, as $\eta \rightarrow \infty$, the upper-bound on the step-size approaches zero. This means that the algorithm is sensitive to both the regularization parameter $\eta$ and the topology (through $\lambda_1(L)$). Lower and upper bounds for $\lambda_1(L)$ can be derived \cite{Grone2}. For example, when the network is fully-connected, it is known that $\lambda_1(L)=N$ and, hence, the bound \eqref{eq:convergence_condition_large_eta} becomes
	\begin{align}
		\mu < \frac{2}{\eta \cdot N}  \label{eq:fully_connected_large_eta_mu_range}
	\end{align}	
	On the other hand, for a network of $N$ agents with maximum degree $\delta$, a lower-bound for the largest eigenvalue of $L$ is \cite{anderson1985eigenvalues}
	\begin{align}
		\lambda_1(L) \geq \frac{N}{N-1} \delta
	\end{align}
	and a necessary (not sufficient) step-size range from \eqref{eq:convergence_condition_large_eta} is
	\begin{align}
	\mu < \frac{2 \cdot (N-1)}{\eta \cdot N\cdot \delta} \label{eq:max_connectivity_large_eta_mu_range}
	\end{align}
	This result implies that as the network size increases (in the case of a fully-connected network) or the connectivity in the network improves ($\delta$ increases), the algorithm becomes less stable (smaller stability range for the step-size is necessary), unlike other distributed algorithms such as consensus or diffusion (Algs.~\ref{alg:Diffusion}--\ref{alg:consensus}).  We conclude from \eqref{eq:convergence_condition_large_eta}, therefore, that in this regime, the stability condition on the step-size is dependent on the network topology.
	
In the following example, we demonstrate a case where the convergence step-size ranges for the AL and AH algorithms are strictly smaller than \eqref{eq:diff_stability_range} so that the stability range for these algorithms can be smaller than non-cooperative agents.

\begin{example}
\label{ex:diverging_example}
\emph{Let $M=1$ and consider the simple $2$-node network illustrated in Fig.~\ref{fig:two_node_network}.
\begin{figure}[H]
	\centering
	\includegraphics[width=0.17\textwidth]{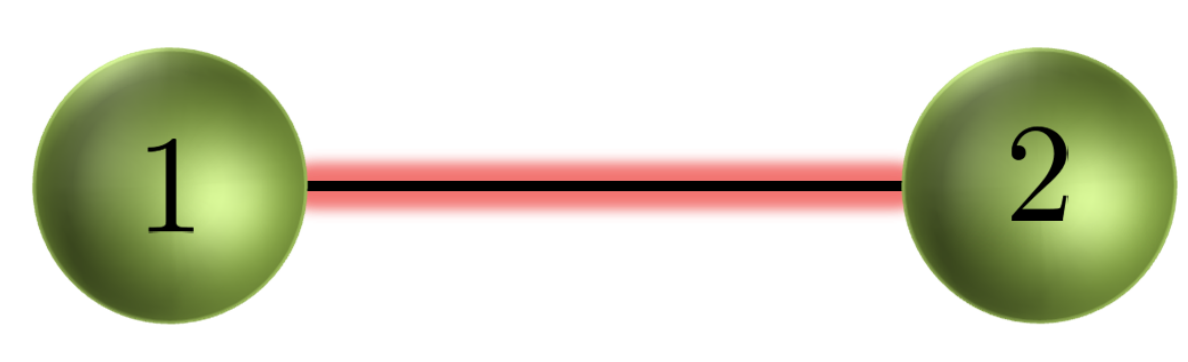}
	\caption{Network topology for Example \ref{ex:diverging_example}.}
	\label{fig:two_node_network}
\end{figure}
\noindent The Laplacian matrix is given by
	\begin{align}
		L = \left[\begin{array}{cc}
		+1 & -1 \\ 
		-1 & +1
		\end{array} \right]
	\end{align}
	The eigenvalues of $L$ are $\{0, 2\}$ and $\mathcal{S}_2 = 1$. Let $R_{u,1} = R_{u,2} = 1$. We consider the AH algorithm. It can be verified that when $\eta = 0$, 
	\begin{align}
		\mathcal{R}' = \left[\begin{array}{rrr}
		1 & 0 & -1 \\ 
		0 & 1 & 0 \\ 
		1 & 0 & 0
		\end{array} \right]
	\end{align}
	so that 
	\begin{align}
		\lambda(\mathcal{B}') = \left\{1\!-\!\mu, 1 \!-\! \frac{\mu}{2} \!+\! \mu\frac{\sqrt{3}}{2} j, 1 \!-\! \mu\frac{1}{2} \!-\! \mu\frac{\sqrt{3}}{2} j\right\}
	\end{align}
	Now assume every agent $k$ runs a non-cooperative algorithm of the following form independently of the other agents
	\begin{align}
		\boldsymbol{w}_{k,i} = \boldsymbol{w}_{k,i-1} + \mu \u_{k,i}^\T (\d_k(i) - \u_{k,i} \boldsymbol{w}_{k,i-1})
	\end{align}
	Then, from \eqref{eq:diff_stability_range}, we know that a sufficient condition on the step-size in order for this non-cooperative solution and for the  diffusion strategy to be mean stable is $0 < \mu < 2$. If we select $\mu=3/2$, then 
	\begin{align}
		\rho(\mathcal{B}') &= \max \left\{\frac{1}{4}, \frac{\sqrt{7}}{2}\right\} = \frac{\sqrt{7}}{2} > 1
	\end{align}
	which implies that the AH algorithm will diverge in the mean error sense. Indeed, it can be verified from \eqref{eq:mu_bar} that the AH algorithm is mean stable when $\mu < 1$. As for the AL algorithm, we may use \eqref{eq:fully_connected_large_eta_mu_range} for the large $\eta$ regime to conclude that the AL algorithm will converge when $\mu < 1/\eta$. If we let $\eta = 20$, then we see that the step-size needs to satisfy $\mu < 1/20 = 0.05$. In Figure \ref{fig:example_divergence}, we simulate this example for $\eta = 20$, and $\mu = 1.1$, $\mu = 0.75$, and $\mu = 3/64 \approx 0.0469$, and set the noise variance for both nodes at $\sigma_v^2 = 0.1$. The consensus algorithm is outperformed by the non-cooperative algorithm when the step-size is large, as predicted by the analysis in \cite{tu2012diffusion}. We observe that, even in this simple fully-connected setting, the AH and AL algorithms are less stable than the diffusion strategies. In fact, the AH and AL algorithms are less stable than the non-cooperative strategy as well. We also observe that the AH and AL algorithms do not match the steady-state MSD performance of the other distributed algorithms. We examine this issue next.} \qed
\end{example}

\begin{figure}[!t]
\centerline{\subfigure[MSD vs. $\eta$.]{\label{fig:eta_tradeoff}\centering \includegraphics[width=0.4\textwidth]{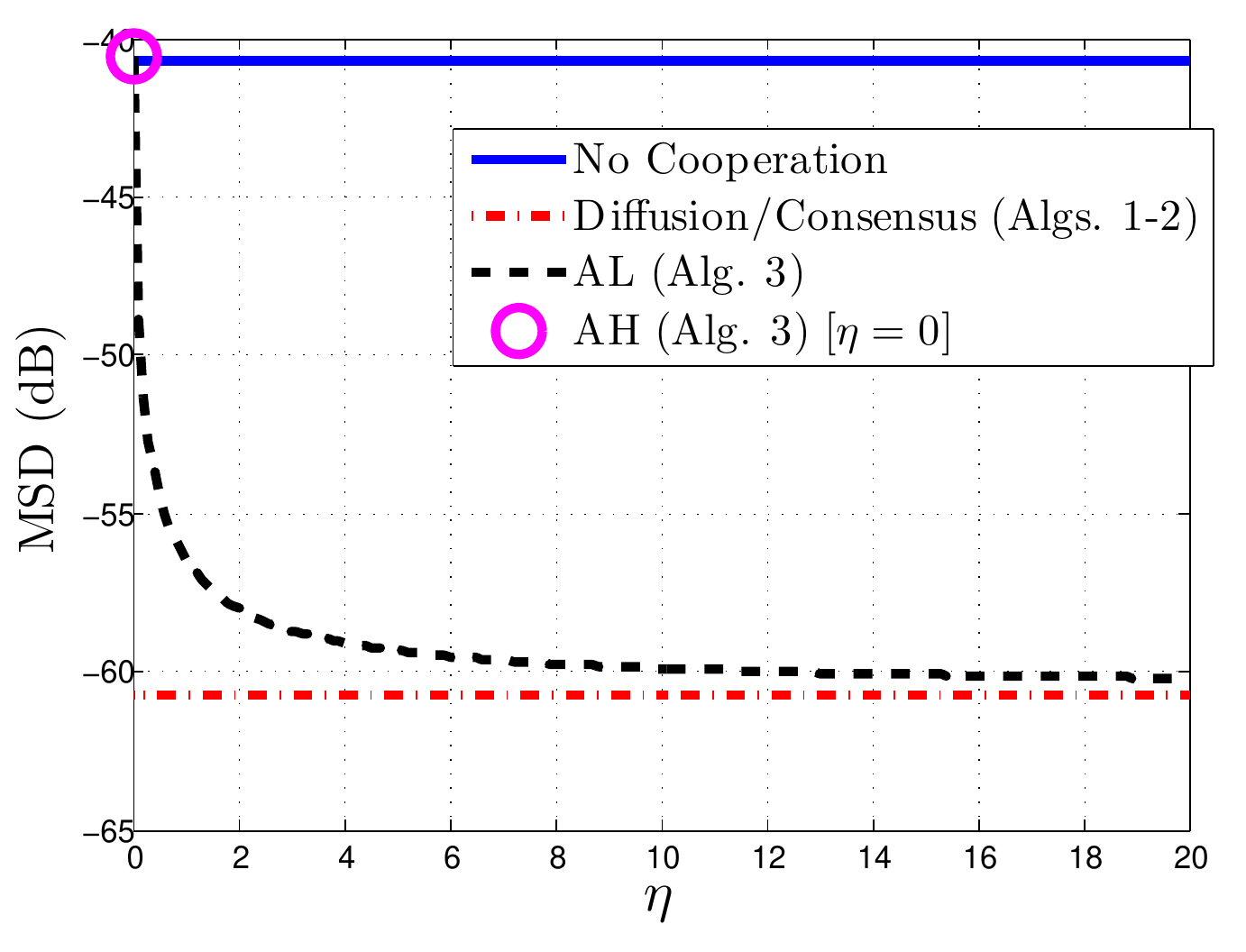}}
}
\centerline{
\subfigure[Algorithm characteristic curves.]{\label{fig:characteristic}\centering \includegraphics[width=0.4\textwidth]{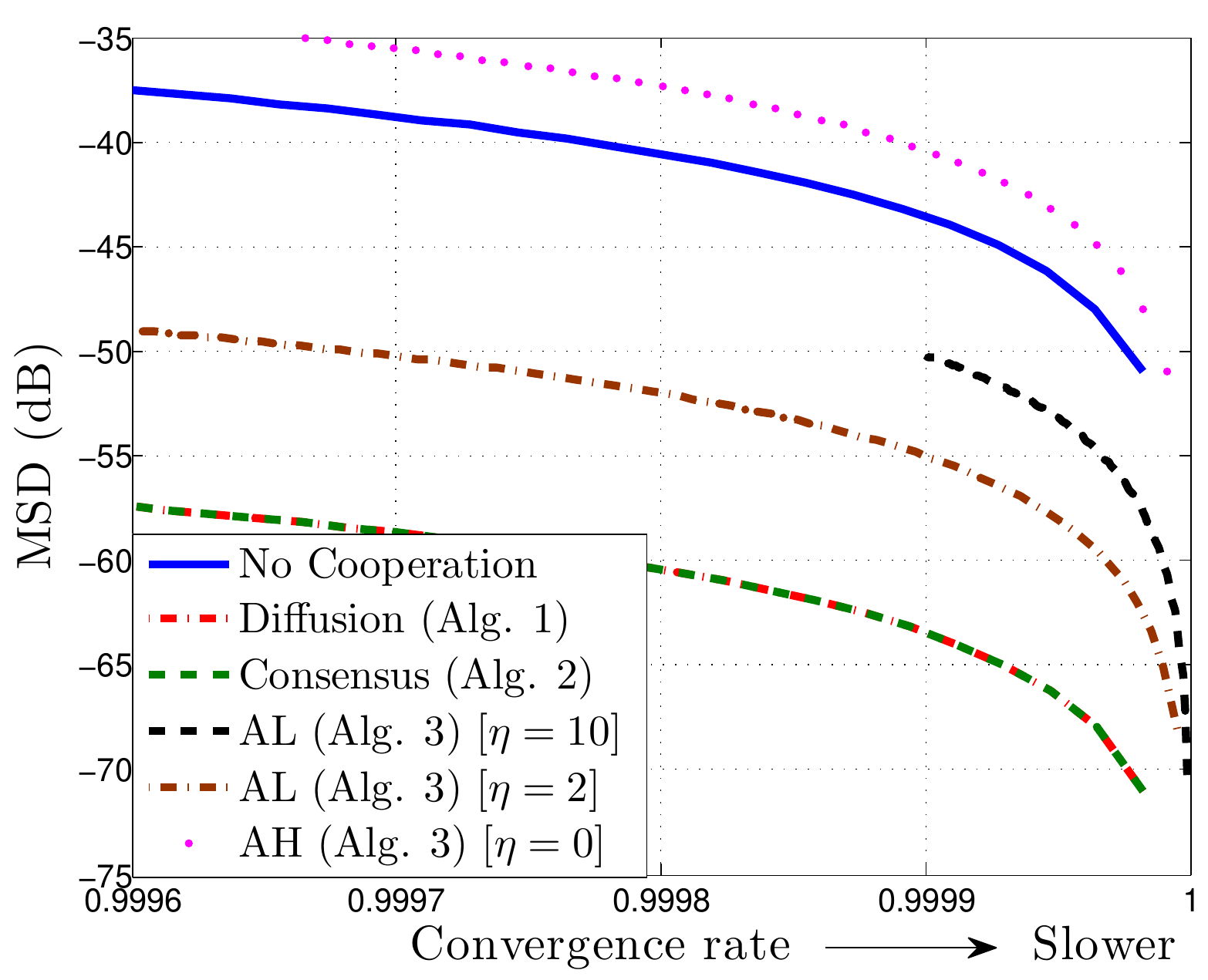}}}
\centerline{
\subfigure[MSD vs. $\eta$ for fixed convergence rate of $0.99995$.]{\label{fig:eta_tradeoff_fixed_convergence_rate}\centering \includegraphics[width=0.4\textwidth]{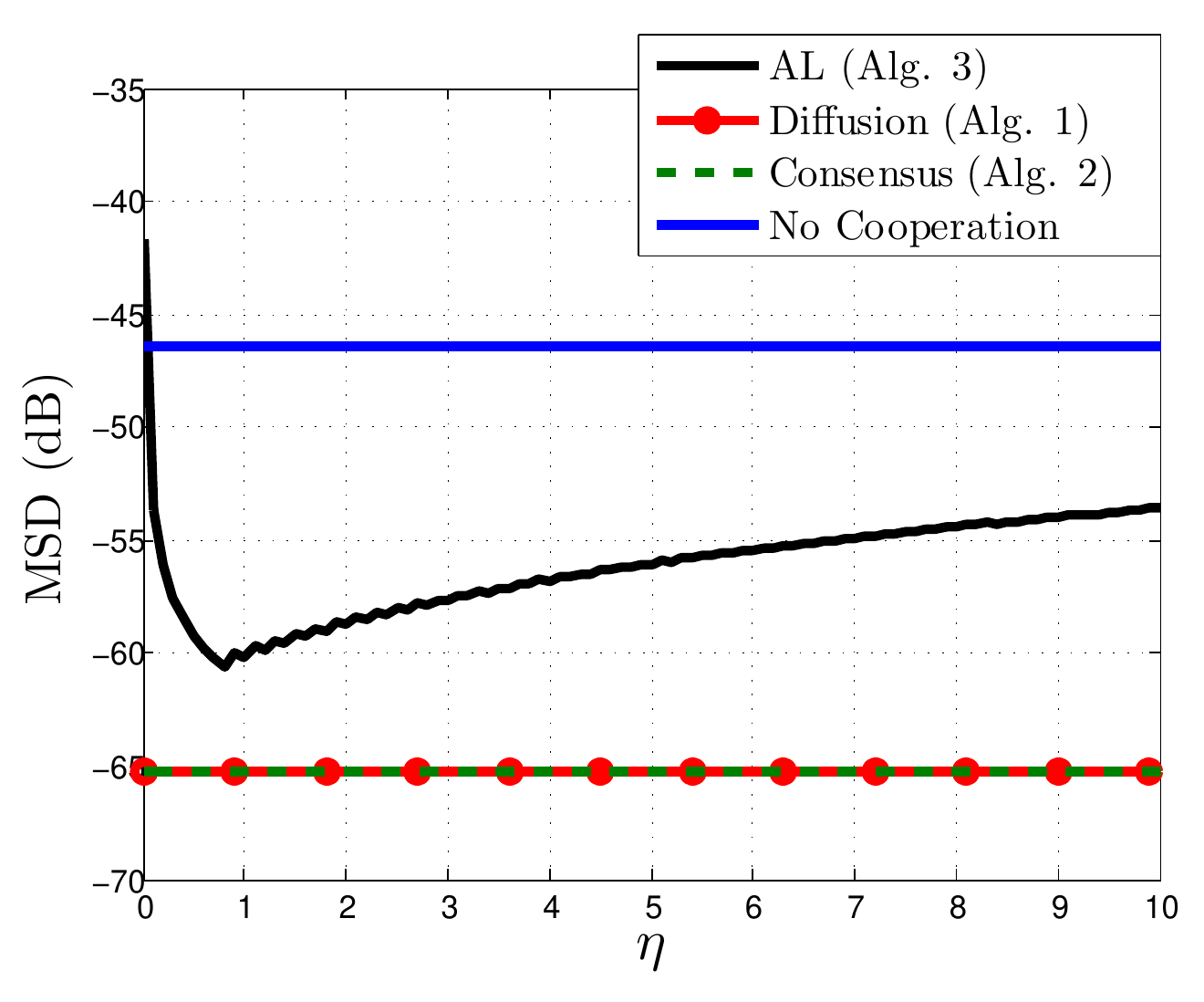}}}
\caption{(a) Performance of the various algorithms against the value of $\eta$, the augmented Lagrangian regularization parameter. $N=100$, $\mu = 0.005$, and $M=5$. All nodes have the same positive-definite matrix $R_u$; (b) Algorithm characteristic curves. Performance of the various algorithms measured in $\rho(\mathcal{B}')$ against the network MSD, curves closer to the bottom-left corner are better. The AL ($\eta = 10$) curve is shorter than the rest since it diverges as $\mu$ is increased to achieve faster convergence; (c) MSD vs. $\eta$ for fixed convergence rate of $0.99995$. Best viewed in color.}
\end{figure}

\section{Mean-Square-Error Performance}
\label{Sec:MSD}
From \eqref{eq:vec_weighted_expression}, we have in the limit that:
	\begin{align}
		\lim_{i\rightarrow\infty} \E \left\|\left[\begin{array}{c}
	\widetilde{\w}_i' \\ 
	\widetilde{\blambda}_{1,i}'
	\end{array}\right]\right\|^2_{(I - \mathcal{F}') \gamma'} &\approx \mu^2 (\mathrm{bvec}(R_h))^\T \gamma' \label{eq:weighted_MSD_expression}
	\end{align}
	Now, since we are interested in the network MSD defined by \eqref{eq:network_MSD}, we may introduce the auxiliary matrix:
	\begin{align}
		\Phi = \left[\begin{array}{c|c}
	I_{NM \times NM} & 0_{NM\times (N-1)M} \\ \hline
	0_{(N-1)M \times NM} & 0_{(N-1)M \times (N-1)M}
	\end{array} \right]
	\end{align}
	and select $\gamma' = (I-\mathcal{F}')^{-1} \mathrm{bvec}(\Phi)$ in \eqref{eq:weighted_MSD_expression} to obtain the following expression for the network MSD:
	 \begin{equation}
	\lim_{i\rightarrow\infty} \E \left\|\widetilde{\w}_i\right\|^2_{\frac{1}{N} I_{NM}} \!\!\approx\!\! \frac{\mu^2}{N} (\mathrm{bvec}(R_h))^\T \!(I \!-\! \mathcal{F}')^{-1} \mathrm{bvec} (\Phi) \label{eq:MSD_SS_approx_network_average}
\end{equation}

\begin{theorem}[MSD approximation for AL and AH algorithms]
	\label{thm:general_MSD_expression}
	Given that the matrix $\mathcal{H} + \eta \mathcal{L}$ is positive-definite from Lemma \ref{lem:pos_def_v'Hv}, then under under sufficiently small step-sizes, the network \emph{MSD} \eqref{eq:MSD_SS_approx_network_average} simplifies to:
	\begin{align}
	\mathrm{MSD} \approx \frac{\mu}{2 N} \mathrm{Tr}\left(\mathcal{R}_z (\mathcal{H} + \eta \mathcal{L})^{-1}\right) + O(\mu^2)\label{eq:most_general_MSD_expression}
	\end{align}
	where ${\cal R}_z$ is given by \eqref{eq:R_z}. 
\end{theorem}
\begin{proof}
	See Appendix \ref{app:proof_general_MSD_expression}.
\end{proof}
Observe that the positive-definiteness of the matrix ${\cal H}+\eta{\cal L}$ is guaranteed for the AH algorithm by assuming that $R_{u,k} > 0$ for all $k$. The following special cases follow from \eqref{eq:most_general_MSD_expression}.
\begin{corollary}[MSD performance of AH algorithm]
	\label{cor:zero_eta_approximation}
	Assuming each $R_{u,k}$ is positive-definite and the network is connected, the network \emph{MSD} \eqref{eq:most_general_MSD_expression} for the AH algorithm is given by:
	\begin{align}
		\mathrm{MSD} &\approx \mu  \frac{ M}{2}  \frac{1}{N}\sum_{k=1}^N \sigma_{v,k}^2  + O(\mu^2) \label{eq:MSD_AH}
	\end{align}
\end{corollary}
\begin{proof}
	The result follows by setting $\eta=0$ in \eqref{eq:most_general_MSD_expression}. 
\end{proof}
\noindent Expression \eqref{eq:MSD_AH} is actually equal to the average performance across a collection of $N$ non-cooperative agents (see, e.g., \cite{sayed2013diffusion,SayedProcIEEE}). In this way, Corollary \ref{cor:zero_eta_approximation} is a surprising result for the AH algorithm since even with cooperation, the AH network is not able to improve over the non-cooperative strategy where each agent acts independently of the other agents. This conclusion does not carry over to the AL algorithm, although the following is still not encouraging for AL strategies. 

\begin{corollary}[MSD performance of AL algorithm]
	\label{cor:general_large_eta_approximation}
	Assume that the matrix $\mathcal{H} + \eta \mathcal{L}$ is positive-definite. Then, for sufficiently small step-sizes, the network \emph{MSD} \eqref{eq:most_general_MSD_expression} for the AL algorithm for large $\eta$ simplifies to 
	\begin{align}
	\mathrm{MSD} &\approx  \frac{\mu}{2N} \Tr\left(\!\!\bigg(\sum_{k=1}^N R_{u,k}\bigg)^{\!\!-1} \!\!\bigg(\sum_{k=1}^N \sigma_{v,k}^2 R_{u,k}\bigg)\!\!\right) \!+ \nonumber\\
	&\quad\ \frac{\mu}{2N\eta} \Tr\left( \mathcal{R}_z \mathcal{L}^{\dagger}\right) \!+O\Big(\frac{\mu}{N^2\eta}\Big) \!+ O(\mu^2) \label{eq:large_eta_MSD_simple_terms}
	\end{align}	
	where $\mathcal{L}^\dagger$ denotes the pseudoinverse of $\mathcal{L}$.
\end{corollary}
\begin{proof}
	See Appendix \ref{app:proof_general_large_eta_approximation}.
\end{proof}
By examining \eqref{eq:large_eta_MSD_simple_terms}, we learn that the performance of the AL algorithm for large $\eta$ approaches the performance of the  the diffusion strategy given by \eqref{eq:diff_cons_MSD}. However, recalling the fact that the step-size range required for convergence, under the large $\eta$ regime and the assumption that $\mathcal{H} = I_N \otimes R_u$ in \eqref{eq:convergence_condition_large_eta}, is inversely proportional to $\eta$, we conclude that the AL algorithm can only approach the performance of the diffusion strategy as $\mu \rightarrow 0$ and $\eta \rightarrow \infty$. In addition, the performance of the AL algorithm depends explicitly on the network topology through the Laplacian matrix $\mathcal{L}$. Observe that this is not the case in \eqref{eq:diff_cons_MSD} for the primal strategies. For this reason, even for large $\eta$, the AL algorithm is not robust to the topology. 

In order to illustrate these results, we consider a connected network with $N=100$ agents and set $\mu = 1\times 10^{-4}$ and $M=5$. First, we illustrate the network MSD \eqref{eq:most_general_MSD_expression} as a function of $\eta$ when $\mathcal{H} = I_N \otimes R_u$ and $R_u$ is positive-definite. We observe from Fig.~\ref{fig:eta_tradeoff} that the MSD performance of the AH algorithm is identical to that of the non-cooperative solution. In addition, the AL algorithm only approaches the same steady-state performance as the diffusion strategy asymptotically as $\eta \rightarrow \infty$. Furthermore, we graph the characteristic curves for various strategies in Fig.~\ref{fig:characteristic}. In this figure, we plot the convergence rate against the MSD \eqref{eq:most_general_MSD_expression}. Clearly, curves closer to the bottom-left corner indicate better performance since in that case an algorithm converges quicker and has better steady-state performance. We notice that the AH algorithm is \emph{outperformed} by non-cooperation. In addition, as $\eta$ increases, we see that the AL algorithm can approach the performance of the diffusion strategy for very small $\mu$ (slow convergence rate---bottom-right part of the plot), but not anywhere else. We also see the effect of \eqref{eq:convergence_condition_large_eta} where the AL algorithm with $\eta = 10$ is less stable than the other distributed algorithms. In Fig.~\ref{fig:eta_tradeoff_fixed_convergence_rate}, we obtain the steady-state MSD values as well as convergence rates for the primal algorithms and non-cooperative algorithm with different step-sizes. At the same time, we obtain the steady-state MSD performance and convergence rate for the AL algorithm with $\eta$ varying from zero to $10$ (inclusive). We then sort the convergence rates and plot the steady-state MSD performance of all algorithms for the convergence rate of $0.99995$. Observe also that the AH algorithm (obtained by setting $\eta = 0$ in the AL algorithm) achieves a \emph{worse} steady-state MSD than the non-cooperative algorithm when the convergence rate is held constant. We note that for this particular choice of covariance matrices, noise variances, and network, the optimal choice for $\eta$ is $\eta = 0.8$.  

It should be noted that the convergence rate of the AH and AL algorithms depends on the network topology through the Laplacian matrix (see Example \ref{ex:diverging_example}). In contast, it was shown in \cite{jianshu_part1} that the \emph{initial} convergence rate of the diffusion and consensus schemes also depends on the network topology. That is, the convergence of the primal schemes exhibits two phases 1) An initial phase that is dependent on the network topology and 2) A second, slower, phase that is \emph{independent} of the network topology. It was further shown that the initial phase is faster than the second phase and thus the convergence rate of the primal strategies is largely independent of the network topology. This is not the case for the AH and AL algorithms. Figure~\ref{fig:characteristic} plots the characteristic curves for the primal algorithms relative to the \emph{slower} (second) convergence rate.

\section{Improving the MSD Performance}

In this section, we set the step-size, $\mu$, as a function of the regularization parameter $\eta$. More specifically, we choose:
\begin{align}
	\eta = \mu^{-1/\theta},\quad\quad\quad [ \theta > 1 ] 
\end{align}
%The AL algorithm \eqref{eq:AL_primal_alg}--\eqref{eq:AL_dual_alg} then leads to:
%\begin{subequations}
%\begin{align}
%\w_i &= \left(I- \mu^{1-\theta} \mathcal{L}\right)\w_{i-1} \!-\! \mu \h_i \!-\! \mu \mathcal{C}^\T \boldsymbol{\lambda}_{i-1} \label{eq:fixed_AL_primal}\\
%	\boldsymbol{\lambda}_i &= \boldsymbol{\lambda}_{i-1} + \mu \mathcal{C} \w_{i-1} \label{eq:fixed_AL_dual}	
%\end{align}
%\end{subequations}
The main difficulty in the stability analysis in this case is that the step-size moves simultaneously with $\eta$. Observe that in our previous results, we found that there exists some step-size range for which the algorithm is stable when $\eta$ is fixed at a large value, and that this step-size bound depends on $\eta$. Unfortunately, we do not know in general how the upper-bound \eqref{eq:mu_bar} depends on $\eta$, and so it is not clear if the algorithm can still be guaranteed to converge when $\mu = \eta^{-\theta}$. If, however, we assume that $\mathcal{H} = I_N \otimes R_u$, then we can obtain the eigenvalues of the matrix $\mathcal{R}'$ using Theorem \ref{thm:eigenvalues_minusR2}, which then allows us to continue to guarantee convergence for large $\eta$.

\begin{theorem}
\label{thm:convergence_fixed_AL}
	Let $\mu = \eta^{-\theta}$ and let $\theta > 1$. Furthermore, let $\mathcal{H} = I_N \otimes R_u$, with positive-definite $R_u$. Then, there exists some positive $\underline{\eta}$ such that for all $\eta > \underline{\eta}$, the matrix $\mathcal{B}'$ is stable; i.e., $\rho(\mathcal{B}') < 1$.
\end{theorem}	
\begin{proof}
	The result follows by using \eqref{eq:eigenvalues_sigma}--\eqref{eq:eigenvalues_tau} to verify that $\eta^{-\theta}$ continues to be upper bounded by \eqref{eq:mu_bar} for large $\eta$.
\end{proof}

\begin{theorem}
		Let $\mu = \eta^{-\theta}$, where $\theta > 1$. Assuming the algorithm converges (guaranteed by Theorem~\ref{thm:convergence_fixed_AL}), then the \emph{MSD} of the modified AL algorithm is approximated by:
		\begin{align}
		\mathrm{MSD} 
		&= \frac{\mu}{2N} \Tr\!\!\left(\!\!\left(\sum_{k=1}^N R_{u,k}\right)^{\!\!\!-1} \!\!\!\!\left(\sum_{k=1}^N \sigma_{v,k}^2 R_{u,k}\right)\!\!\right) \!+\! O\left(\!\mu^{1+1/\theta}\!\right)
		\end{align}
		for large $\eta$, or small $\mu$.
\end{theorem}
\begin{proof}
		Substitute $\eta = \mu^{-\nicefrac{1}{\theta}}$ into the results of Corollary \ref{cor:general_large_eta_approximation}.
\end{proof}
\noindent The drawback remains that the AL algorithm is less stable than primal-optimization techniques, such as the diffusion strategy.

\section{Numerical Results}
Consider a network of $N=20$ agents and $M=5$. We generate a positive-definite matrix $R_u > 0$ with eigenvalues $1 + \x_m$, where $\x_m$ is a uniform random variable. We let $\mathcal{H} = I_N \otimes R_u$ with $\mu = 0.01$. This value allows all algorithms to converge. The diffusion and consensus strategies utilize a doubly-stochastic matrix generated through the Metropolis rule \cite{sayed2012diffbookchapter}. We note that the diffusion and consensus algorithms can improve their MSD performance by designing the combination matrix based on the Hastings rule \cite{zhao2012performance,boyd2004fastest}, but we assume that the nodes are noise-variance agnostic. In Fig.~\ref{fig:simulation1}, we simulate Algorithms~\ref{alg:Diffusion}--\ref{alg:AL}, and for Algorithm \ref{alg:AL}, we simulate three  values of $\eta$: $0$, $0.2$, and $2$. This will allow us to validate our analysis results where an increase in $\eta$ yields improvement in the MSD (see Corollary \ref{cor:general_large_eta_approximation}). The theoretical curves are generated using \eqref{eq:most_general_MSD_expression}. We observe that as $\eta$ is increased, the performance of the AL algorithm improves, but is still worse than that of the consensus algorithm \eqref{eq:consensus_1}--\eqref{eq:consensus_2} and the diffusion strategy \eqref{eq:diff_A}--\eqref{eq:diff_C}. Furthermore, as indicated by Fig.~\ref{fig:characteristic}, the convergence rate of the AH algorithm is worse than that of non-cooperation, even though both algorithms achieve the same MSD performance. This was observed earlier in Fig.~\ref{fig:example_divergence}. It is possible to further increase $\eta$ in order to make the performance of the AL algorithm match better with that of the consensus and diffusion strategies. However, it is important to note that if $\eta$ is increased too much, the algorithm will diverge (recall \eqref{eq:convergence_condition_large_eta}). For this reason, it is necessary to find a compromise between finding a large enough $\eta$ that the network MSD would be small, and a small enough $\eta$ to enhance \eqref{eq:convergence_condition_large_eta}.

\begin{figure}[!th]
	\centering
	\includegraphics[width=0.43\textwidth]{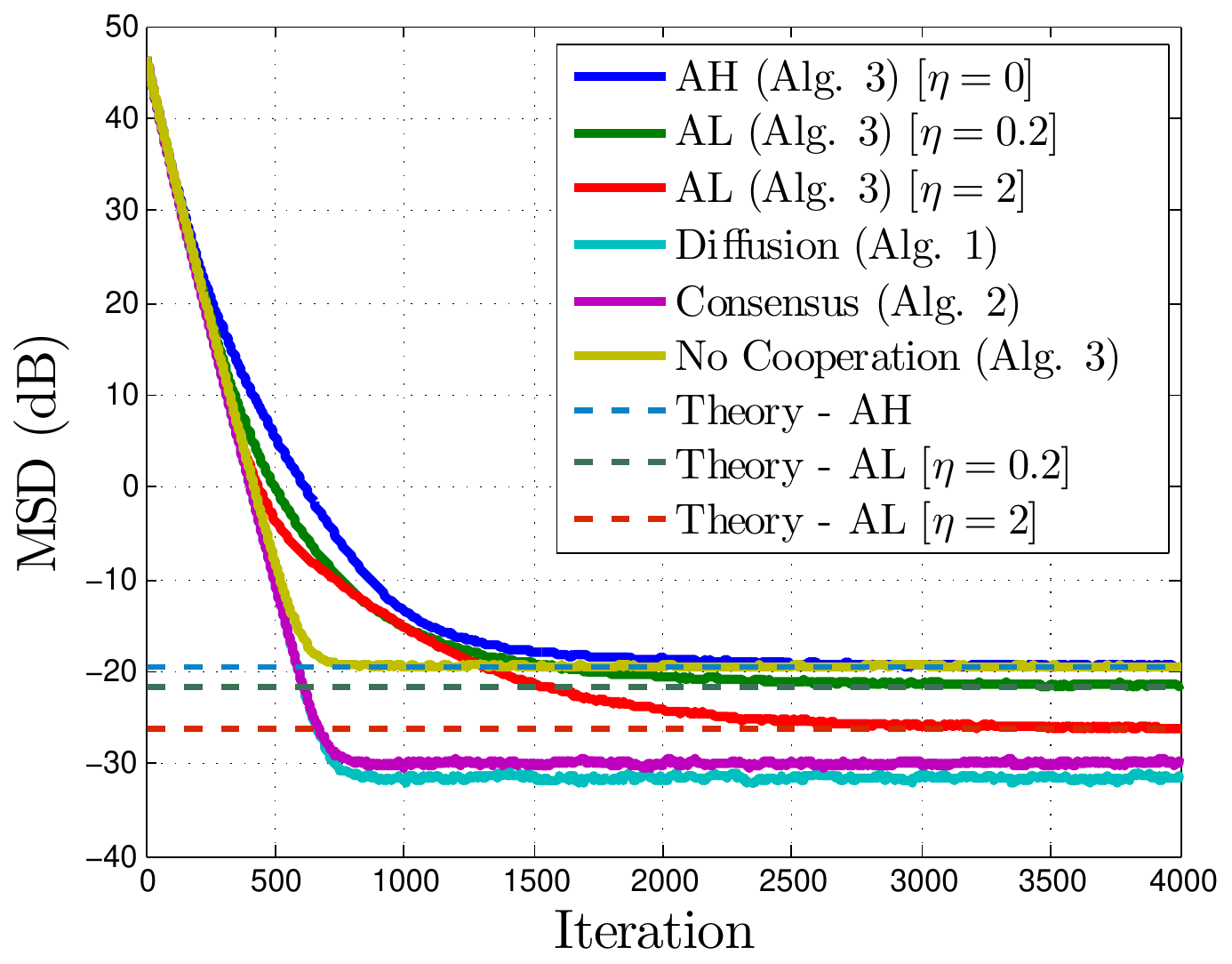}
	\caption{Simulation result for Algs.~\ref{alg:Diffusion}--\ref{alg:AL}. Best viewed in color.}
	\label{fig:simulation1}
\end{figure}

%\vspace*{-1\baselineskip}

\section{Conclusion}

In this work, we examined the performance of primal-dual methods in a stochastic setting. In particular, we analyzed the performance of a first-order AH algorithm and a first-order AL algorithm. We discovered that the performance of the AH algorithm matches that of a non-cooperative solution and has stability limitations. We also showed that the AL algorithm can asymptotically match the MSD performance of the diffusion algorithm when the regularization parameter is increased. Unfortunately, as $\eta$ is increased, we saw that the permissible step-size range for the algorithm to converge shrinks. We provided a ``fix'' for the AL algorithm where we link the step-size to the regularization parameter $\eta$. With this modification, we show that the performance of the AL method matches that of the diffusion and consensus strategies up to first order in  the step-size. Unfortunately, this change does not remedy the fact that the step-size range for stability is still more restrictive than that of other distributed methods.

% or


\appendices

\section{The AH Strategy under Partial Observation}
\label{app:AH_partial_observability}
In this appendix we provide an example to illustrate that the AH strategy can become unstable under the partial observation setting when some of the individual covariance matrices are singular. Thus, consider a fully-connected three node network with the incidence matrix
\begin{align}
	C = \left[\begin{array}{rrr}
	1 & -1 & 0 \\ 
	1 & 0 & -1 \\ 
	0 & 1 & -1
	\end{array} \right]
\end{align}
\noindent Furthermore, let $R_{u,1} = \mathrm{diag}\{1,0,0\}$, $R_{u,2} = \mathrm{diag}\{0,1,0\}$, and $R_{u,3} = \mathrm{diag}\{0,0,1\}$. Observe that the covariance matrices satisfy the condition $R_{u,1} + R_{u,2} + R_{u,3} > 0$, even though each $R_{u,k}$ is singular. Now, computing the SVD of the incidence matrix $\mathcal{C}$, we obtain
\begin{align}
	S_2 = \left[\begin{array}{rr}
	\sqrt{3} & 0\\ 
	0 & \sqrt{3}
	\end{array} \right] ,\quad V = \left[\begin{array}{rrr}
	-\frac{1}{\sqrt{2}} & -\frac{1}{\sqrt{6}} & \frac{1}{\sqrt{3}} \\ 
	0 & \sqrt{\frac{2}{3}} & \frac{1}{\sqrt{3}} \\ 
	\frac{1}{\sqrt{2}} & -\frac{1}{\sqrt{6}} & \frac{1}{\sqrt{3}}
	\end{array} \right]
\end{align}
Then,
\begin{align}	
-\mathcal{R}' &= -\left[\begin{array}{ccc}
\mathcal{V}_2^\T \mathcal{H} \mathcal{V}_2 & \mathcal{V}_2^\T \mathcal{H} \mathcal{V}_0 & \mathcal{S}_2^\T \\ 
\mathcal{V}_0^\T \mathcal{H} \mathcal{V}_2 & \displaystyle \frac{1}{3} \sum_{k=1}^3 R_{u,k} & 0_{3 \times 6}\\
-\mathcal{S}_2 & 0_{6 \times 3} & 0_{6}
\end{array} \right]
\end{align}
where $\mathcal{H} = \mathrm{diag}\{R_{u,1},R_{u,2},R_{u,3}\}$. It is straightforward to verify that the spectrum of $-\mathcal{R}'$ contains a purely imaginary eigenvalue at $j\sqrt{3}$. For example, if we let $v = \frac{1}{\sqrt{2}} [0, j, 0, 0, 0, 0, 0, 0, 0, 0, 1, 0, 0, 0, 0]^\T$, then $-\mathcal{R}' v = j\sqrt{3} v$, 
which implies that $-\mathcal{R}'$ is \emph{not} Hurwitz. Therefore, we cannot find a positive range for the step-size $\mu$ to guarantee convergence of the AH algorithm according to Lemma \ref{lem:discrete-time-stability} even though the AL algorithm can be guaranteed to converge for large $\eta$ (see Theorem \ref{thm:generalTopology_mean_stability_partial_observability}) and the diffusion and consensus strategies can also be shown to converge in this setting \cite{NOW_ML,SayedProcIEEE}. In Figure \ref{fig:simulation_partial_observation}, we simulate the described scenario with $\mu = 0.02$ and $\sigma_{v,k}^2 = 0.01$ for all $k=1,2,3$. We see that while the AH algorithm oscillates, the other algorithms converge.
\begin{figure}
\centering
\includegraphics[width=0.43\textwidth]{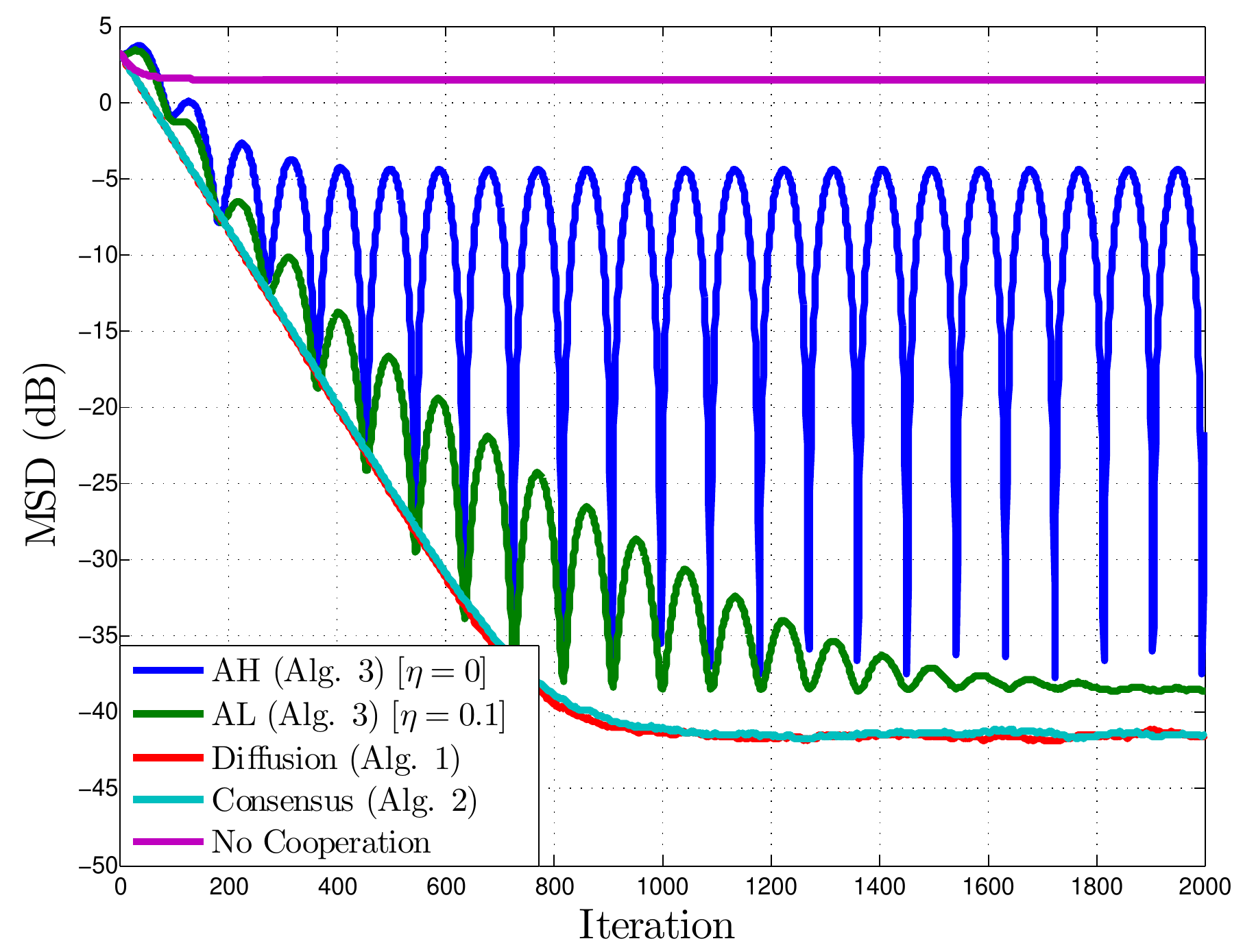}
\caption{Simulation result for Algs. \ref{alg:Diffusion}--\ref{alg:AL} for the partial observation scenario described in App.~\ref{app:AH_partial_observability}. The curves are averaged over $1000$ experiments.}
\label{fig:simulation_partial_observation}
\end{figure}
This result is surprising since the network is \emph{fully connected} and yet the AH algorithm cannot estimate the desired parameter vector. Indeed, for this example, AH algorithm will not converge for any $\mu > 0$. The curves are averaged over $1000$ experiments.

\section{Proof of Theorem \ref{thm:general_MSD_expression}}
\label{app:proof_general_MSD_expression}
%First, we begin from the steady-state approximation expression \eqref{eq:MSD_SS_approx_network_average}
%\begin{align}
%&\lim_{i\rightarrow\infty} \E \left\|\widetilde{\w}_i\right\|^2_{\frac{1}{N} I_{NM}} \approx \frac{\mu^2}{N} \left[\begin{array}{c}
%\mathrm{vec}(\mathcal{V}^\T R_z \mathcal{V}) \\ 
%\mathrm{vec}(0_{(N-1)M\times NM}) \\ 
%\mathrm{vec}(0_{NM \times (N-1)M}) \\ 
%\mathrm{vec}(0_{(N-1)M\times (N-1)M})
%\end{array} \right]^\T \times \nonumber\\
%&\quad\quad(I - \mathcal{F}')^{-1} \left[\begin{array}{c}
%\mathrm{vec}(I_{MN}) \\ 
%\mathrm{vec}(0_{(N-1)M\times NM}) \\ 
%\mathrm{vec}(0_{NM \times (N-1)M}) \\ 
%\mathrm{vec}(0_{(N-1)M\times (N-1)M})
%\end{array}\right] \label{eq:analysis_blocked}
%\end{align}
We refer to expression \eqref{eq:MSD_SS_approx_network_average}. We know that the matrix $I - \mathcal{F}'$ is stable when the step-size is small by Theorem~\ref{thm:generalTopology_mean_square_stability_partial_observability}. Hence, 
\begin{equation}
	(I \!-\! \mathcal{F}')^{-1} \!=\! I\!+\!{\cal F}'\!+\!({\cal F}')^2+\ldots \!=\! \sum_{n=0}^\infty \!\mathcal{B}'^{n \T} \otimes_b \mathcal{B}'^{n \T}
\end{equation}
Then, using properties of block Kronecker products \cite{NOW_ML}:
\begin{equation}
	\left(\mathrm{bvec}(R_h)\!\right)^\T\!\!(I \!-\! \mathcal{F}')^{-1} \mathrm{bvec}\!\left(\Phi\right)\! =\! \sum_{n=0}^\infty\!\Tr\!\left(\!R_h \mathcal{B}'^{n\T} \Phi \mathcal{B}'^{n}\!\right)
\end{equation}
Now, observe that for small step-sizes we have, using \eqref{eq:R'_mean_stability},
\begin{align}
	\mathcal{B}'^{n \T} \!&= \!(I\!-\!\mu \mathcal{R}'^{\T})^n \!=\!\! \left[\!\!\begin{array}{cc|c} \multicolumn{2}{c|}{\multirow{2}{*}{$I-\mu \mathcal{K}$}} & \mu \mathcal{S}_2^\T\\
													  \multicolumn{2}{c|}{}  & 0_{M\!\times\!(N-1)M}\\\hline
													   -\mu \mathcal{S}_2 & 0_{(N-1)M\!\times\!M} & I_{(N-1)M}  \end{array}\!\!\right]^{\!n} \nonumber\\
						&\approx \left[\!\!\begin{array}{cc|c} \multicolumn{2}{c|}{\multirow{2}{*}{$(I-\mu \mathcal{K})^n$}} & \mu n \mathcal{S}_2^\T\\
													  \multicolumn{2}{c|}{}  & 0_{M\!\times\!(N-1)M}\\\hline
													   -\mu n \mathcal{S}_2 & 0_{(N-1)M\!\times\!M} & I_{(N-1)M}  \end{array}\!\!\right]
%						\left[\begin{array}{cc}
%						(I-\mu \mathcal{K})^n & \mu n \mathcal{S}_2^\T \\ 
%						-\mu n \mathcal{S}_2 & I
%						\end{array} \right]
\end{align}
where $\mathcal{K} \triangleq \mathcal{V}^\T \mathcal{H} \mathcal{V} + \eta \sLambda$ and therefore, ignoring higher-order powers of the small step-size parameter:
\begin{align}
	 &\mathcal{B}'^{n\T} \Phi \mathcal{B}'^{n} \!\approx\!  \left[\!\!\begin{array}{cc|c} \multicolumn{2}{c|}{\multirow{2}{*}{$(I-\mu \mathcal{K})^n$}} & \mu n \mathcal{S}_2^\T\\
													  \multicolumn{2}{c|}{}  & 0\\\hline
													   -\mu n \mathcal{S}_2 & 0 & I_{(N-1)M}  \end{array}\!\!\right] \!\!\left[\begin{array}{cc}
						I & 0 \\ 
						0 & 0
						\end{array} \right] \times\nonumber\\
						&\left[\!\!\begin{array}{cc|c} \multicolumn{2}{c|}{\multirow{2}{*}{$(I\!-\!\mu \mathcal{K})^n$}} & -\mu n \mathcal{S}_2^\T\\
													  \multicolumn{2}{c|}{}  & 0\\\hline
													   \mu n \mathcal{S}_2 & 0 & I_{(N-1)M}  \end{array}\!\!\right] 
	\!\!\approx\!\!  \left[\!\!\begin{array}{cc|c} \multicolumn{2}{c|}{\multirow{2}{*}{$(I\!-\!\mu \mathcal{K})^{2n}$}} & -\mu n \mathcal{S}_2^\T\\
													  \multicolumn{2}{c|}{}  & 0\\\hline
													   -\mu n \mathcal{S}_2 & 0 & 0  \end{array}\!\!\right]
	\end{align}
Collecting these results we get, since $\mathcal{K} > 0$,
\begin{align}
	&(\mathrm{bvec}(R_h))^\T (I\!-\!\mathcal{F}')^{-1} \mathrm{bvec}(\Gamma) \!\approx\! \sum_{n=0}^\infty \Tr(\mathcal{V}^\T \!\mathcal{R}_z\! \mathcal{V} (I \!-\! \mu \mathcal{K})^{2n})\nonumber\\
	&\approx \sum_{n=0}^\infty \!\Tr(\mathcal{V}^\T \mathcal{R}_z \mathcal{V} (I \!-\! 2\mu \mathcal{K})^{n}) \!=\! \Tr(\mathcal{V}^\T \!\mathcal{R}_z \!\mathcal{V} (2\mu \mathcal{K})^{-1})
\end{align}
\vspace*{-1\baselineskip}
\section{Proof of Corollary \ref{cor:general_large_eta_approximation}}
\label{app:proof_general_large_eta_approximation}
%\begin{figure*}[!b]
%\normalsize
%% IEEE uses as a separator
%\hrulefill
%% The spacer can be tweaked to stop underfull vboxes.
%\vspace*{4pt}
%% Store the current equation number.
%\setcounter{mytempeqncnt}{\value{equation}}
%% Set the equation number to one less than the one
%% desired for the first equation here.
%% The value here will have to changed if equations
%% are added or removed prior to the place these
%% equations are referenced in the main text.
%\setcounter{equation}{178}
%\begin{align}
%	&\mathcal{K}^{-1} \approx \left[\!\!\!\begin{array}{cc}
%	  \frac{1}{\eta} \mathcal{D}_1^{-1} & -\frac{1}{\eta}\mathcal{D}_1^{-1} \mathcal{V}_2^\T \mathcal{H} \mathcal{V}_0 {\bar{R}_u}^{-1}  \\ 
%	-\frac{1}{\eta} {\bar{R}_u}^{-1} \mathcal{V}_0^\T \mathcal{H} \mathcal{V}_2 \mathcal{D}_1^{-1}  & {\bar{R}_u}^{-1} + \frac{1}{\eta} {\bar{R}_u}^{-1} \mathcal{V}_0^\T \mathcal{H} \mathcal{V}_2 \mathcal{D}_2^{-1} \mathcal{V}_2^\T \mathcal{H} \mathcal{V}_0{\bar{R}_u}^{-1}
%	\end{array} \!\!\!\right] \label{eq:general_inverse_script_K}
%\end{align}
%% Restore the current equation number.
%\setcounter{equation}{\value{mytempeqncnt}}
%\end{figure*}
\noindent Using \eqref{eq:y'Hy} we have
%\begin{align}
%	&\mathcal{V}^\T \mathcal{H} \mathcal{V} = \left[\begin{array}{c}
%	V_2^\T \otimes I_M \\ 
%	\frac{1}{\sqrt{N}}\mathds{1}^\T \otimes I_M
%	\end{array} \right] \mathcal{H} \left[\begin{array}{cc}
%	V_2 \otimes I_M & \frac{1}{\sqrt{N}} \mathds{1} \otimes I_M
%	\end{array} \right] \nonumber\\
%	&\!\!=\!\! \left[\!\!\!\begin{array}{cc}
%	(V_2^\T \otimes I_M)\mathcal{H} ( V_2 \otimes I_M ) &  \frac{1}{\sqrt{N}} (V_2^\T \otimes I_M) \mathcal{H} (\mathds{1} \otimes I_M) \\ 
%	(\frac{1}{\sqrt{N}}\mathds{1}^\T \otimes I_M) \mathcal{H} (V_2 \otimes I_M)  & \frac{1}{N} \sum_{k=1}^N R_{u,k}
%	\end{array} \!\!\!\right]
%\end{align}
%where the singular vectors $V$ were partitioned according to \eqref{eq:subdivided_V}. Now, recall from \eqref{eq:Lambda} that $\mathcal{S}_1^\T \mathcal{S}_1$ contains the eigenvalues of the matrix $\mathcal{L}$:
%\begin{align}
%\mathcal{S}_1^\T \mathcal{S}_1 = \Lambda \otimes I = \left[\begin{array}{cc}
%\Lambda_1 \otimes I & 0 \\ 
%0 & 0_{M\times M}
%\end{array} \right]
%\end{align}
%where the $(N-1)\times (N-1)$ matrix $\Lambda_1$ contains the positive eigenvalues of the Laplacian matrix $L$. It follows that
\begin{align}
	{\cal V}^{\sf T} \left({\cal H}+\eta{\cal L}\right) {\cal V} &= \left[\!\!\begin{array}{cc}
	\mathcal{V}_2^\T \mathcal{H} \mathcal{V}_2 + \eta \mathcal{D}_1& \mathcal{V}_2^\T \mathcal{H} \mathcal{V}_0 \\ 
	\mathcal{V}_0^\T \mathcal{H} \mathcal{V}_2 & \displaystyle \frac{1}{N} \sum_{k=1}^N R_{u,k}
	\end{array} \!\!\right] \label{eq:matrix_form_V'HV+e*S_1'S_1}
\end{align}
where $\sLambda_1 = \mathcal{S}_2^\T \mathcal{S}_2 = D_1 \otimes I_M$. The invertibility of the above matrix is guaranteed by a large enough $\eta$ (see Lemma \ref{lem:pos_def_v'Hv}). We may now use the block matrix inversion formula:
\begin{align}
	\left[\!\!\!\begin{array}{cc}
	A & B \\ 
	C & D
	\end{array} \!\!\!\right]^{-1} \!=\! \left[\!\!\!\begin{array}{cc}
	E  & -E B D^{-1} \\ 
	-D^{-1} C E & D^{-1} + D^{-1} C E B D^{-1}
	\end{array} \!\!\!\right] \label{eq:block_matrix_inverse}
\end{align}
where $E = (A - B D^{-1} C)^{-1} \approx \frac{1}{\eta} \mathcal{D}_1^{-1}$ for large $\eta$. Defining 
\begin{align}
	\bar{R}_u \triangleq \frac{1}{N} \sum_{k=1}^N R_{u,k},\quad\quad \bar{R}_z \triangleq \frac{1}{N} \sum_{k=1}^N R_{z,k}
\end{align}
and applying \eqref{eq:block_matrix_inverse} to \eqref{eq:matrix_form_V'HV+e*S_1'S_1}, we obtain 
\begin{align}
&{\cal V}^\T \!({\cal H}\!+\!\eta {\cal L})^{-1}\!{\cal V}  \approx \nonumber\\
&\!\left[\!\!\!\!\begin{array}{cc}
	  \frac{1}{\eta} \mathcal{D}_1^{-1} & -\frac{1}{\eta}\mathcal{D}_1^{-1} \mathcal{V}_2^\T \mathcal{H} \mathcal{V}_0 {\bar{R}_u}^{-1}  \\ 
	-\frac{1}{\eta} {\bar{R}_u}^{-1} \mathcal{V}_0^\T \mathcal{H} \mathcal{V}_2 \mathcal{D}_1^{-1}  & {\bar{R}_u}^{-1} \!\!+\!\! \frac{1}{\eta} {\bar{R}_u}^{-1} \mathcal{V}_0^\T \mathcal{H} \mathcal{L}^\dagger \!\mathcal{H} \mathcal{V}_0{\bar{R}_u}^{-1}
	\end{array} \!\!\!\!\right] \label{eq:general_inverse_script_K}
\end{align}
%\eqref{eq:general_inverse_script_K} at the bottom of the page. \addtocounter{equation}{1}
We also have that:
\begin{align}
	\mathcal{V}^\T \mathcal{R}_z \mathcal{V} &= \left[\!\!\begin{array}{c}
	\mathcal{V}_2^\T \\ 
	\mathcal{V}_0^\T 
	\end{array} \!\!\right]  \!\mathcal{R}_z \!\left[\!\!\begin{array}{cc}
	\mathcal{V}_2 &\!\!\!\! \mathcal{V}_0
	\end{array}\!\! \right] = \!\!\left[\!\!\!\begin{array}{cc}
	\mathcal{V}_2^\T \mathcal{R}_z\! \mathcal{V}_2 &\!\! \mathcal{V}_2^\T  \mathcal{R}_z \mathcal{V}_0  \\ 
	\mathcal{V}_0^\T \mathcal{R}_z\! \mathcal{V}_2  &\!\! \bar{R}_z
	\end{array} \!\!\!\right] \label{eq:general_VT_Rz_V}
\end{align}
Substituting \eqref{eq:general_inverse_script_K}--\eqref{eq:general_VT_Rz_V} into \eqref{eq:most_general_MSD_expression}, we have that the performance of the algorithm, for large $\eta$, is given by \eqref{eq:large_eta_MSD_simple_terms}.

\ifCLASSOPTIONcaptionsoff
  \newpage
\fi

\bibliographystyle{IEEEtran}
\bibliography{refs}

\end{document}